\documentclass[12pt]{amsart}
\usepackage{amsmath}
\usepackage{amscd}
\usepackage{amssymb}
\usepackage{amsfonts}
\usepackage{amsthm}
\usepackage{bbm}
\usepackage{cancel}
\usepackage{color}
\usepackage{eucal}
\usepackage{enumerate,yfonts}
\usepackage{epstopdf}
\usepackage{epsfig}
 \usepackage{float}
\usepackage{pdfsync}
 \usepackage[all,cmtip]{xy}
 \usepackage[all]{xy}
\usepackage{graphicx}
\usepackage{graphics}
\usepackage{hyperref}
\usepackage{latexsym}
\usepackage{mathrsfs}
 \usepackage{placeins}
\usepackage{pstricks}

\usepackage{stmaryrd}
\usepackage{url}
\usepackage{xy}
\usepackage{xypic}

\newtheorem{thm}{Theorem}[section]

\newtheorem{lemma}[thm]{Lemma}

\newtheorem{proposition}[thm]{Proposition}
\newtheorem{prop}[thm]{Proposition}
\newtheorem{conjecture}[thm]{Conjecture}
\newtheorem{thm-dfn}[thm]{Theorem-Definition}
\newtheorem*{Assumption}{Assumption}

\theoremstyle{definition}

\newtheorem{remark}[thm]{Remark}
\newtheorem{example}[thm]{Example}

\numberwithin{equation}{section}

\newcommand{\fg}{{\mathfrak g}}

\newcommand{\fn}{{\mathfrak n}}
\newcommand{\fp}{{\mathfrak p}}

\newcommand{\fa}{{\mathfrak a}}

\newcommand{\fc}{{\mathfrak c}}

\newcommand{\fF}{{\mathfrak{F}}}

\newcommand{\Lg}{{\mathfrak g}}

\newcommand{\bC}{{\mathbb C}}

\newcommand{\bG}{{\mathbb G}}
\newcommand{\bZ}{{\mathbb Z}}

\newcommand{\bbP}{{\mathbb P}}

\newcommand{\bP}{{\mathbb P}}

\newcommand{\Hess}{\on{Hess}}

\newcommand{\mY}{\mathcal{Y}}

\newcommand{\mE}{\mathcal{E}}
\newcommand{\mF}{\mathcal{F}}

\newcommand{\mO}{\mathcal{O}}
\newcommand{\mL}{\mathcal{L}}
\newcommand{\mH}{\mathcal{H}}

\newcommand{\mN}{\mathcal{N}}

\newcommand{\calF}{{\mathcal F}}

\newcommand{\calL}{{\mathcal L}}

\newcommand{\calS}{{\mathcal S}}

\newcommand{\cO}{{\mathcal O}}

\newcommand{\cF}{{\mathcal F}}
\newcommand{\cN}{{\mathcal N}}

\newcommand{\cE}{{\mathcal E}}

\newcommand{\cM}{{\mathcal{M}}}

\newcommand{\on}{\operatorname}

\newcommand{\tX}{\widetilde X}

\newcommand{\tQ}{\widetilde Q}
\newcommand{\tE}{\widetilde E}

\newcommand{\tV}{\widetilde V}

\newcommand{\ra}{\rightarrow}

\newcommand{\is}{\simeq}

\newcommand{\Loc}{\on{LocSys}}

\newcommand{\tC}{\widetilde C}

\newcommand{\nc}{\newcommand}

\nc{\al}{{\alpha}} \nc{\be}{{\beta}} \nc{\ga}{{\gamma}}
\nc{\ve}{{\varepsilon}} \nc{\Ga}{{\Gamma}} 
\nc{\La}{{\Lambda}}
\nc{\ad }{{\on{ad }}}

\nc{\aff}{{\on{aff}}} \nc{\Aff}{{\mathbf{Aff}}}

\nc{\der}{{\on{der}}}

\nc{\diag}{{\on{diag}}}

\nc{\Fl}{{\calF\ell}}

\nc{\Hg}{{\on{Higgs}}}
\newcommand{\Hom}{{\on{Hom}}}

\nc{\Id}{{\on{Id}}}

\nc{\Ind}{{\on{Ind}}}
\newcommand{\Lie}{{\on{Lie}}}
\nc{\Op}{{\on{Op}}}

\nc{\res}{{\on{res}}}

\nc{\tr}{{\on{tr}}}

\nc{\GSp}{{\on{GSp}}} \nc{\GU}{{\on{GU}}} \nc{\SL}{{\on{SL}}}
\nc{\SU}{{\on{SU}}} \nc{\SO}{{\on{SO}}}

\newcommand{\ars}{{\mathfrak a^{rs}}}
\newcommand{\grs}{{\mathfrak g_1^{rs}}}

\nc{\nh}{{\Loc_{J^p}(\tau')}}
\nc{\bnh}{{\Loc_{\breve J^p}(\tau')}}

\nc{\bU}{{\overline{U}}} 
\nc{\IC}{{\on{IC}}}

\newcommand{\p}{\perp}

\nc{\ot}{\otimes}

\nc{\oh}{{\operatorname{H}}}
\nc{\gr}{{\operatorname{gr}}}
\nc{\rk}{{\operatorname{rank}}}
\nc{\codim}{{\operatorname{codim}}}
\nc{\img}{{\operatorname{Im}}}
\nc{\Span}{{\operatorname{Span}}}
\nc{\Img}{\operatorname{Im}}

\newcommand{\beqn}{\begin{equation*}}
\newcommand{\eeqn}{\end{equation*}}

\newcommand{\beq}{\begin{equation}}
\newcommand{\eeq}{\end{equation}}

\newcommand{\bern}{\begin{eqnarray*}}
\newcommand{\eern}{\end{eqnarray*}}

\usepackage{etoolbox}
\makeatletter
\patchcmd{\@maketitle}
  {\ifx\@empty\@dedicatory}
  {\ifx\@empty\@date \else {\vskip1ex\footnotesize \centering\@date\par\vskip1ex}\fi
   \ifx\@empty\@dedicatory}
  {}{}
\patchcmd{\@adminfootnotes}
  {\ifx\@empty\@date\else \@footnotetext{\@setdate}\fi}
  {}{}{}
\makeatother

\begin{document}
\title[Hessenberg varieties and the Springer correspondence]{Hessenberg varieties, intersections of quadrics, and the Springer correspondence}
        \author{Tsao-Hsien Chen} 
        \thanks{Tsao-Hsien Chen was supported in part by the AMS-Simons travel grant.}
        \address{Department of Mathematics, University of Chicago, Chicago, 60637, USA}
         
        \email{chenth@math.uchicago.edu}
        \author{Kari Vilonen}
        \thanks{Kari Vilonen was supported in part by the ARC grants DP150103525 and DP180101445, the Academy of Finland, NSF grant DMS-1402928, the Humboldt Foundation, and the Simons Foundation.}

        \address{School of Mathematics and Statistics, University of Melbourne, Australia and Department of Mathematics and Statistics, University of Helsinki, Finland}
         \email{kari.vilonen@unimelb.edu.au}
\author{Ting Xue}
\address{
School of Mathematics and Statistics, University of Melbourne, Australia and Department of Mathematics and Statistics, University of Helsinki, Finland}
\email{ting.xue@unimelb.edu.au}

\thanks{Ting Xue was supported in part by the ARC grants DP150103525, DE160100975 and the Academy of Finland.}

\begin{abstract}
In this paper we introduce a certain class of families of Hessenberg varieties 
arising from Springer theory for symmetric spaces.
We study the geometry of those Hessenberg varieties
and investigate their monodromy representations 
in detail using the geometry of complete intersections of quadrics. 
We obtain decompositions of these monodromy representations into irreducibles
 and 
compute the Fourier transforms of the IC complexes
associated to these irreducible representations. 
The results of the paper refine (part of) the Springer correspondece 
for the split symmetric pair $(SL(N),SO(N))$
in \cite{CVX2}.

\end{abstract}

\maketitle

\setcounter{tocdepth}{1} \tableofcontents

\section{Introduction}
In this paper we study the geometry of Hessenberg varieties of \cite{GKM} 
arising from 
Springer theory for symmetric spaces \cite{CVX,CVX2}.
Let us recall the set-up. Let $G$ be a reductive group and $\theta$ an involution of $G$. We write $K= (G^\theta)^0$ for the connected component of the fixed point set. This gives rise to a symmetric pair $(G,K)$. We also have the corresponding decomposition of the Lie algebra $\fg=\fg_0\oplus\fg_1$ where $\fg_0$ is the fixed point set and $\fg_1$ is the $(-1)$-eigenspace of $\theta$, respectively. We write $\cN_1=\cN\cap\fg_1$ where $\cN$ is the nilpotent cone in $\fg$. 
Let us call an irreducible $K$-equivaraint IC-sheaf\footnote{IC=intersection cohomology} 
supported on the nilpotent cone $\mathcal N_1$
a \emph{nilpotent orbital complex}.
We address the following question which can be regarded as an analogue of the classical springer correspondence:  what are the Fourier transforms of 
nilpotent orbital complexes? 

We work in the context of the split symmetric pair $(G,K)=(SL(N),SO(N))$ where $K=G^\theta$ is given by an involution $\theta:G\to G$ and $N$ is odd. 
Recall that the $K$-orbits in $\cN_1$ are parametrized by partitions of $N$. 
In \cite{CVX2}, we show that the Fourier transform
gives a bijection between nilpotent orbital complexes and certain representations 
of braid groups. We identify these representations of braid groups and constructed them explicitly using representations of Hecke algebra of symmetric groups at $q=-1$. 
This bijection can be viewed as Springer correspondence for the symmetric pair 
$(SL(N),SO(N))$. In the course of the proof we give a construction of nilpotent orbital complexes 
with full support Fourier transforms, following Lusztig we call them \emph{thick}
nilpotent orbital complexes, using Grinberg's
nearby cycle sheaves.

The goal of this paper is to 
use the geometry of Hessenberg varieties to
reach a better understanding of 
Fourier transforms of nilpotent orbital complexes and 
Springer correspondence 
for the symmetric pair 
$(SL(N),SO(N))$. We concentrate on the thick nilpotent orbital complexes because 
in \cite[Corollary 4.8]{CVX2} we show that one can obtain all nilpotent orbital complexes by induction 
from thick nilpotent orbital complexes of smaller groups.

To this end we proposed in~\cite{CVX} a general method of analyzing Fourier transforms of 
nilpotent orbital complexes. We replace the Springer resolution and the Grothendieck simultaneous resolution of the classical Springer correspondence by (several) pairs of families of Hessenberg varieties  $\mathcal{X}$ and  $\check{\mathcal{X}}$ and obtain the following picture:
 \beq
\label{advanced springer diagram}
\begin{CD}
\mathcal{X}  @>>> \check{\mathcal{X}}
\\
@V{\pi}VV  @VV{\check\pi}V
\\
\cN_1  @>>>  \fg_1
\end{CD}
\eeq
The image of $\pi$ is a nilpotent orbit closure $\bar\cO$ but neither $\pi$ nor $\check \pi$ are semi-small in general. In fact, their generic fibers are not just points in general but they form smooth families of varieties.  The key to analyzing the Fourier transforms of 
nilpotent orbital complexes
 in this manner is that the constant sheaf on $\check{\mathcal{X}}$ is the Fourier transform of the constant sheaf on $\mathcal{X}$. Thus, at least as a first approximation, we are reduced to decomposing the push-forwards $\pi_*\bC_ \mathcal{X} $ and $\check\pi_*\bC_{\check{\mathcal{X}}}$ into direct sums of IC-sheaves; this is possible by the decomposition theorem. 
In \cite{CVX} we study the case when the thick nilpotent orbital complexes are supported on nilpotent orbits of order 2, i.e., orbits which correspond to partitions that only involve 2's and 1's,
and we show that the relevant families of Hessenberg varieties are closely related to 
Jacobians of hyperelliptic curves (see \cite[\S 3]{CVX}). 
In this paper we treat the case of orbits of order 3. It turns out that the 
relevant Hessenberg varieties are closely related to complete intersection of quadrics.

We now describe our results in more details.
Let us recall the definition of 
Hessenberg varieties in our setting following \cite{GKM}. Let $x\in\fg_1$, $P$ a parabolic subgroup of $K$, and $\Sigma\subset \fg_1$ a $P$-invariant subspace. The Hessenberg variety associated to the triple $(x,P,\Sigma)$, denoted by 
$\on{Hess}_x(K/P,\fg_1,\Sigma)$, is by definition the 
following variety \[\on{Hess}_x(K/P,\fg_1,\Sigma):=\{g\in K/P\,|\,g^{-1}x\in \Sigma\}.\]
As $x$ varies over $\fg_1$, we get a family of Hessenberg varieties 
$\on{Hess}(K/P,\fg_1,\Sigma)\ra\fg_1$.

The particular pairs of families of Hessenberg varieties we study here have the following properties. 
One of the families in the pair, when restricted to the regular semi-simple locus, is  isomorphic to a family
of complete intersections of quadrics (see Theorem \ref{H=X}); this is the family $\check\pi: \check{\mathcal{X}}\to  \fg_1$ in~\eqref{advanced springer diagram}. The other family, corresponding to $\pi: {\mathcal{X}}\to  \cN_1$ in~\eqref{advanced springer diagram}, is supported on the locus of nilpotent elements of order at most  three and the fibers of this family admit affine pavings (see \S\ref{affine pavings}).

The main results in this paper are in \S\ref{sec-monodromy of Hess} and \S\ref{computation of FT}.  When restricted to the locus of regular semi-simple elements $\fg_1^{rs}$ our particular families of Hessenberg varieties can be interpreted as families of intersections of quadrics in projective spaces. In \S\ref{sec-monodromy of Hess},
we study the monodromy representations arising from the primitive cohomology of these families (and their natural double covers). This is accomplished by establishing a relative version of  results of T. Terasoma \cite{T} in \S\ref{generalization of T}. 
The resulting monodromy representations of $\pi_1^{K}(\fg_1^{rs})$ can be expressed in terms of 
monodromy representations of  certain families of hyperelliptic curves over $\fg_1^{rs}$.
In particular, we see that the cohomology of those Hessenberg varieties can be 
expressed in terms of cohomology of hyperelliptic curves. 
In Theorems~\ref{irred} and~\ref{irred tE} we describe these monodromy representations completely by decomposing them into irreducible pieces which we call $E_{ij}^N$ and $\tE_{ij}^N$, respectively. In \S\ref{computation of FT}, we study
Fourier transforms of the IC complexes arising from the local systems $E_{ij}^N$ and $\tE_{ij}^N$. Recall that these 
were obtained from the primitive cohomology of the particular families of 
Hessenberg varieties and their double covers. 
We show that their Fourier transforms are 
supported on the closed sub-variety $\mathcal N_1^3\subset\mathcal N_1$ consisting of 
nilpotent elements of order less than or equal to three. In particular, 
we obtain various examples of thick nilpotent orbital complexes.
Let $\{(\mO,\mE)\}_{\leq 3}$ be the set of all pairs $(\mO,\mE)$ where $\mO$ is a 
$K$-orbit in $\mathcal N_1^3$ and $\mE$ is an irreducible $K$-equivariant local system on 
$\mO$ (up to isomorphism). In this manner we obtain 
an injective map
\beq\label{injection1}
\{E_{ij}^N\}\cup\{\tE_{ij}^N\}\hookrightarrow\{(\mO,\mE)\}_{\leq 3}\,.
\eeq
This map refines (part of) the Springer correspondence in \cite[Theorem 4.1]{CVX2}.
We would like to emphasis that such a refinement of
Springer correspondence is crucial for applications, for example, computing cohomologies of 
Fano varieties in 
\cite{CVX3}.

As an interesting corollary 
(see Example \ref{hyperelliptic curve of odd}), we show that the 
Fourier transform of the $\IC$ complex for the unique non-trivial irreducible $K$-equivariant local system on
the minimal nilpotent orbit 
 has full support  and the corresponding local system is given by
the monodromy representation of the universal family of hyperelliptic curves
of genus $n$, where $2n+1=N$.

The paper is organized as follows. In \S\ref{Hessenberg varieties}, we 
introduce certain pairs of families of Hessenberg varieties and prove basic facts about them. 
In \S\ref{generalization of T} and \S\ref{sec-monodromy of Hess}, we establish a relative version of the results of Terasoma~\cite{T}. We utilize these results to obtain a 
decomposition of the monodromy representations  into irreducibles.
In \S\ref{computation of FT}, using the results in previous sections,
we show that 
Fourier transforms of the IC complexes 
for the local systems arising from
our families of 
Hessenberg varieties and their double covers are 
supported on the closed sub-variety $\mathcal N_1^3\subset\mathcal N_1$ consisting of 
nilpotent elements of order at most  three. 
In \S\ref{conjs and examples} we give a conjectural (explicit) description of the map in 
\eqref{injection1} (see Conjecture \ref{conj 1} and Conjecture \ref{conj 2}) for $E_{ij}^N$ and we verify 
the conjectures in various examples.

{\bf Acknowledgement.} We thank Cheng-Chiang Tsai and 
Zhiwei Yun for helpful discussions.  
KV and TX also thank Manfred Lehn, Anatoly S. Libgober, and Yoshinori Namikawa for helpful discussions.
We thank the Max Planck Institute for Mathematics in Bonn and the Mathematical Sciences Research Institute in Berkeley for support, hospitality, and a nice research environment. Furthermore KV and TX thank the Research Institute for Mathematical Sciences in Kyoto  for support, hospitality, and a nice research environment. We also thank the referee for carefully reading our paper.

\section{Hessenberg varieties}\label{Hessenberg varieties}
In this section we introduce certain families of Hessenberg varieties 
which  naturally arise when computing the Fourier transforms of IC complexes supported on 
nilpotent orbits of order less than or equal to three, i.e., orbits of the form $\cO_{3^i2^j1^k}$. 
Our main theorem 
(see Theorem \ref{H=X}) says that, generically, these families of Hessenberg varieties
are isomorphic to families of complete intersections of quadrics.

\subsection{Hessenberg varieties}
Let $G$ be a reductive group and  $V$  a representation of $G$.
 Let $P\subset G$ be a parabolic subgroup  and $\Sigma\subset V$ a $P$-invariant subspace. Consider the vector bundle $G\times^P\Sigma$ and write
\beq
G\times^P\Sigma = \on{Hess}(G/P,V,\Sigma) \ra G/P.
\eeq
The natural projection to $V$ gives us a projective morphism
\beq
\on{Hess}(G/P,V,\Sigma) \ra V\,.
\eeq
The fibers of this morphism are called  \emph{Hessenberg varieties}; the fiber over $v$ is given by 
\beq \on{Hess}_v(G/P,V,\Sigma)=\{gP\,|\,g^{-1}v\in \Sigma\}.
\eeq

Consider now the situation when we have a connected reductive group $G$ and an involution $\theta:G\to G$. Let $K=G^\theta$. The involution $\theta$ induces a grading $\Lg=\Lg_0\oplus\Lg_1$ on the Lie algebra $\Lg$ of $G$, where $\Lg_i=\{x\in\Lg\,|\,d\theta(x)=(-1)^ix\}$. The group $K$ acts on $\Lg_1$ by adjoint action. An element $x_0$ in $\Lg_1$ is said to be regular if $\dim Z_K(x)\geq\dim Z_K(x_0)$ for all $x\in\Lg_1$. We write $\Lg_1^{rs}$ for the set of regular semisimple elements in $\Lg_1$.

Let $T_K$ be a maximal torus of $K$ and consider a co-character $\lambda:\bG_m\ra T_K$. We write $P=P(\lambda)$ for the parabolic subgroup of $K$ associated  to $\lambda$, $\fp$ for the Lie algebra of $P$, 
and $\fg_1=\bigoplus\fg_{1,j}$ for the grading induced by $\lambda$.
For any $i\in\bZ$ we define $\fg_{1,\geq i}=\bigoplus_{j\geq i}\fg_{1,j}$.
Let $\Sigma\subset\fg_1$ be a $P$-invariant subspace. The Hessenberg varieties that we are concerned with are of the form $\Hess_v(K/P,\fg_1,\Sigma)$. We have:
\begin{lemma}[\cite{GKM}]
Suppose $\Sigma\supset\fg_{1,\geq i}$ for some $i\leq 0$.
Then the projective morphism $\Hess(K/P,\fg_1,\Sigma)\ra\fg_1$ is smooth over 
$\fg_1^{rs}$, the set of regular semisimple elements in $\Lg_1$.
\end{lemma}
\begin{proof}
This is proved in \cite[\S 2.5]{GKM}. For the reader's convenience, we recall the argument here.
Observe that the Zariski tangent space to 
$\Hess_v(K/P,\fg_1,\Sigma)\subset K/P$ at a point 
$x=kP\in\Hess_v(K/P,\fg_1,\Sigma)$
can be identified with the kernel of 
\[[v,-]:T(K/P)|_{x}\cong K\times^P(\fg_0/\fp)|_{x}\ra K\times^P(\fg_1/\Sigma)|_{x},\,(k,w)\mapsto (k,[k^{-1}v,w]). \]
So it suffices to show that the map above is surjective on the fibers at each point 
$kP\in\Hess_v(K/P,\fg_1,\Sigma)$ if $v\in\fg_1^{rs}$. For this we show that any 
$v^*\in\fg_1^*$ that annihilates both $[k^{-1}v,\fg_0]$ and $\Sigma $ is zero. 
Since $v^*$ annihilates $\Sigma$ and $\Sigma\supset\fg_{1,\geq i}$ for some $i\leq 0$, there 
exists $\delta>0$ such that $v^*\in\fg_{1,\geq\delta}^*$, that is, $v^*$ is $K$-unstable.  
Since $k^{-1}v\in\fg_1^{rs}$, there is no non-zero $K$-unstable vector $v^*\in\fg_1^*$ that annihilates  the subspace $[k^{-1}v,\fg_0]\subset \fg_1$, we have $v^*=0$. The lemma is proved.
\end{proof}

\subsection{Families of Hessenberg varieties}\label{two Hessenberg}
From now on we concentrate on the following symmetric pair.  
Let $G=SL(N,\bC)$ and $\theta:G\to G$ the involution such that $K:=G^\theta=SO(N,\bC)$. The pair $(G,K)$ is called a split symmetric pair. As in \cite{CVX}, we will assume, starting with~\S\ref{curves}, that $N=2n+1$ is odd, mainly for simplicity.  

Let us write $(G,K)=(SL(V),SO(V,Q))$, where $Q$ is a non-degenerate quadratic form on $V$. Denote by $\langle,\rangle_Q$ the non-degenerate bilinear form associated to $Q$. For a subspace $U\subset V$, we write $U^\p=\{v\in V\,|\,\langle v,U\rangle_Q=0\}$.

Let $\cN$ be the nilpotent cone of $\Lg$ and let $\cN_1=\Lg_1\cap\cN$.  It is known that the number of $K$-orbits in $\cN_1$ is finite (see \cite{KR}). Moreover, the $K$-orbits in $\cN_1$ are parametrized as follows (see \cite{S}). For $N$ odd (resp. even), each partition of $N$ correspond to one $K$-orbit in $\cN_1$ (resp. except that each partition with only even parts corresponds to two $K$-orbits).  In this paper we do not distinguish the two orbits corresponding to the same partition when $N$ is even; thus we write $\cO_\lambda$ for an orbit corresponding to $\lambda$.

Let $\{e_i,\,i=1,\ldots,N\}$ be a basis  of $V$ such that $\langle e_i,e_j\rangle_Q=\delta_{i+j,N+1}$. For any $l\leq \frac{N}{2}$, 
let $P_{l}$ be the parabolic subgroup of $K$ that stabilizes the 
partial flag 
\[0\subset V_{l-1}^0\subset V_{l}^0\subset V_l^{0,\p}\subset V_{l-1}^{0,\p}\subset V=\bC^{N},\]
where $V_i^0=\on{span}\{ e_1,...,e_i\}$. Consider the following two subspaces of 
$\fg_1$,
$$E_{l}=\{x\in\fg_1\,|\,xV_{l}^0=0,\, xV_{l}^{0\p}\subset
V_{l-1}^0\}\text{ and }O_l=\{x\in\fg_1\,|\,xV_{l}^0=0,\, xV_{l-1}^{0\p}\subset
V_{l-1}^0\}.$$ Note that both $E_l$ and $O_l$ are $P_l$-invariant. We form the 
corresponding families of Hessenberg varieties 
\beq\label{Hess E}
\tau_l^N:\Hess_l^E:=\Hess(K/P_l,\fg_1,E_l)\ra\fg_1
\eeq
\beq\label{Hess N}
\sigma_l^N:\Hess_l^O:=\Hess(K/P_l,\fg_1,O_l)\ra\fg_1\,.
\eeq
A direct calculation shows that 
\begin{subequations}
\beq\label{support tau}
\on{Im}\,\tau_l^{N}=\bar{\mathcal O}_{3^{l-1}2^11^{N+1-3l}}\text{ if }3l\leq N+1,\ \ \\
\on{Im}\,\tau_l^{N}=\bar{\mO}_{3^{N-2l}2^{3l-N}}\text{ if }3l>N+1
\eeq
and 
\beq\label{support sigma}
\on{Im}\,\sigma_l^{N}=\bar{\mathcal O}_{3^{l-1}1^{N+3-3l}}\text{ if }3l\leq N+1,\ 
\on{Im}\,\sigma_l^{N}=\bar{\mO}_{3^{N-2l}2^{3l-N}}\text{ if }3l>N+1.
\eeq
\end{subequations}
\begin{remark}
When $3l\leq N+1$, $\tau_l^N$ coincides with 
Reeder's resolution of $\bar{\mathcal O}_{3^{l-1}2^11^{N+1-3l}}$ \cite{R}.
\end{remark}

Let $E_l^\p$ and $O_l^\p$ be the orthogonal complements of $E_l$ and $O_l$ 
in $\fg_1$ with respect to the non-degenerate trace form, respectively. Let us now consider the following families of Hessenberg varieties 
\beq
\check\tau_l^N:{\Hess_l^{E,\p}}:=\Hess(K/P_l,\fg_1,E_l^\p)\ra\fg_1
\eeq
\beq
\check\sigma_l^N:{\Hess_l^{O,\p}}:=\Hess(K/P_l,\fg_1,O_l^\p)\ra\fg_1\,.
\eeq
Concretely, we have 
\[E_l^\p=\{x\in\fg_1\,|\,xV_{l-1}^0\subset V_l^0,\, xV_l^0\subset V_{l}^{0\p}\},\ \ 
O_l^\p=\{x\in\fg_1\,|\,xV_{l-1}^0\subset V_l^0\};\]
and then
\[
{\Hess_l^{O,\p}}\is\{(x,0\subset V_{l-1}\subset V_l\subset V_l^\p\subset V_{l-1}^\p\subset \bC^{N})\,|\,x\in\fg_1,\,
xV_{l-1}\subset V_{l}\},\]
\[
{\Hess_l^{E,\p}}\is\{(x,0\subset V_{l-1}\subset V_l\subset V_l^\p\subset V_{l-1}^\p\subset \bC^{N})\,|\,x\in\fg_1,
\,xV_{l-1}\subset V_{l},\, xV_{l}\subset V_l^\p\}.
\]
Finally, note that, as the notation indicates, the bundle $\Hess_l^{E,\p}\to K/P_l$ is the orthogonal complement of the bundle $\Hess_l^{E}\to K/P_n$ in the trivial bundle $\fg_1 \times K/P_l$ and similarly for ${\Hess_l^{O,\p}}$ and ${\Hess_l^{O}}$.
Hence, by functoriality of the Fourier transform, we have:
\beq
\label{Fourier}
\fF((\sigma^N_{l})_*\bC[-])\cong(\check{\sigma}^N_{l})_*\bC[-]\text{ and }\fF((\tau^N_{l})_*\bC[-])\cong(\check{\tau}^N_{l})_*\bC[-].
\eeq

\subsection{Affine pavings}\label{affine pavings}
 In this subsection we show that 
 \beq
 \begin{gathered}
\text{ the fibers of $\tau_m^N:\on{Hess}_m^{E}\to\cN_1$ and $\sigma_m^N:\on{Hess}_m^{O}\to\cN_1$  have}\\\text{ a paving by affine spaces.}
\end{gathered}
 \eeq

\begin{lemma}\label{lemma-fiberiso}
Let $x\in\cO_{3^i2^j1^{N-3i-2j}}\subset\on{Im}\tau_m^N$ (resp. $\on{Im}\sigma_m^N$) and $x_0\in\cO_{2^j1^{N-3i-2j}}$.
We have 
\beqn
(\tau_m^{N})^{-1}(x)\cong(\tau_{m-i}^{N-3i})^{-1}(x_0)\ (\text{resp. }(\sigma_m^{N})^{-1}(x)\cong(\sigma_{m-i}^{N-3i})^{-1}(x_0)).
\eeqn
\end{lemma}
\begin{proof}
We prove the lemma for $\tau_m^N$. The argument for $\sigma_m^N$ is entirely similar and omitted. We have
$$(\tau_m^{N})^{-1}(x)\cong\{0\subset V_{m-1}\subset V_m\subset V_m^\p\subset V_{m-1}^\p\subset\bC^{N}\,|\,xV_m=0,\, xV_{m}^\p\subset V_{m-1}\}.$$
Let $(V_{m-1}\subset V_m)\in(\tau_m^{N})^{-1}(x)$. We have that $\on{Im}\,x\subset(\ker x)^\p\subset V_m^\p$. Thus $\on{Im}\,x^2\subset V_{m-1}$.

Choose a basis $\{x^ku_l,\, k\in[0,2],\, l\in[1,i],\ v_k,\,xv_k, \,k\in[1,j],\ w_l,\, l\in[1,N-3i-2j]\}$ of $V$ as in \cite[Lemma 5.6]{CVX}. 
Let $$U^0=\on{span}\{x^lu_k,l\in[0,2],k\in[1,i]\},\,V^0=\on{span}\{v_k,xv_k,k\in[1,j],w_l,\ l\in[1,N-3i-2j]\}.$$ Then $Q|_{U^0}$, $Q|_{V^0}$ are non-degenerate and $V^0=(U^0)^\p$.   We have
$$V_m=\on{Im}\,x^2\oplus W_{m-i}\text{ and }V_{m-1}=\on{Im}\,x^2\oplus W_{m-i-1},$$ where $W_{m-i}=V_m\cap V^0\supset W_{m-i-1}=V_{m-1}\cap V^0$. We have $\dim W_{m-i}=m-i$ and $\dim W_{m-i-1}=m-i-1$. Let $x_0=x|_{V^0}$. Then $x_0\in\cO_{2^j1^{N-3i-2j}}$. Note that $xV_m=0$ if and only if $x_0W_{m-i}=0$. Now $$V_m^\p=\on{span}\{x^lu_k,l=1,2,k\in[1,i]\}\oplus W_{m-i}^{\p_0},$$
where $W_{m-i}^{\p_0}$ denotes the orthogonal complement of $W_{m-i}$ in $V^0$ with respect to $Q|_{V^0}$. Thus $xV_m^\p\subset V_{m-1}$ if and only if $x_0W_{m-i}^{\p_0}\subset W_{m-i-1}$. This gives us the desired isomorphism
\beqn
(\tau_m^{N})^{-1}(x)\cong(\tau_{m-i}^{N-3i})^{-1}(x_0),\ (V_{m-1},V_m)\mapsto(\on{pr_{V^0}}(V_{m-1}),\on{pr_{V^0}}(V_{m}))
\eeqn
where $\on{pr_{V^0}}$ is the projection from $V$ to $V^0$ with respect to $V=U^0\oplus V^0.$
\end{proof}

Let $\on{OGr}(k,N)$ denote the orthogonal Grassmannian variety of $k$-dimensional isotropic subspaces in $\bC^N$ with respect to a non-degenerate bilinear form on $\bC^N$ and $\on{Gr}(k,N)$ denote the Grassmannian variety of $k$-dimensional subspaces in $\bC^N$.

By Lemma \ref{lemma-fiberiso}, to describe the fibers $(\tau_m^N)^{-1}(x)$ and $(\sigma_m^N)^{-1}(x)$, it suffices to consider the case when $x\in\cO_{2^j1^{N-2j}}$.  
We first introduce some notation. Let $x\in\cO_{2^j1^{N-2j}}$. We write $$\Sigma:=\ker x/\on{Im}x\text{ and }\bar{U}=U/(U\cap\on{Im}x)\text{ for }U\subset\ker x.$$ Define a bilinear from $(,)$ on $\on{Im} x$ by 
\beq\label{bilinear}
(xv,xw):=\langle v,xw\rangle_Q.\eeq
Using $(\on{Im}x)^\p=\ker x$, we see that $(,)$ is non-degenerate. For $U\subset\on{Im}x$, we define $$U^{\perp_{(,)}}=\{u\in\on{Im}x\,|\,(u,U)=0\}.$$
Let us denote 
$$\Upsilon_{m,j}^N:=(\tau_m^N)^{-1}(x),\ \ \Gamma_{m,j}^N:=(\sigma_m^N)^{-1}(x),\ \ x\in\cO_{2^j1^{N-2j}}.$$

 We partition $\Upsilon_{m,j}^N$ into pieces indexed by the dimension of $V_m\cap \operatorname{Im}x$ by letting 
\beqn
\Upsilon_{m,j}^{N,k}=\{(0\subset V_{m-1}\subset V_m\subset V_m^\p\subset V_{m-1}^\p\subset\bC^{N})\in\Upsilon_{m,j}^N\,|\,\dim(V_m\cap \operatorname{Im}x)=k\}.
\eeqn
We now describe the pieces $\Upsilon_{m,j}^{N,k}$. To this end, let
\beqn
\begin{gathered}
\Theta_{m,j}^{N,k}=\{0\subset V_m\subset V_m^\p\subset\bC^N\,|\,\dim(V_m\cap \operatorname{Im}x)=k,
\\
\hspace{.6in}xV_m=0,\, xV_m^\p\subset V_m,\, \dim(xV_m^\p)\leq m-1\}.
\end{gathered}
\eeqn
Consider the following map
\beqn
\eta:\Theta_{m,j}^{N,k}\to\on{Gr}(j-k,\on{Im}\,x)\times\on{Gr}(m-k,\Sigma),\ (V_m)\mapsto((V_m\cap\on{Im}\,x)^{\p_{(,)}},\bar{V}_m).
\eeqn
 We claim that
\beqn
\on{Im}\eta\cong\on{OGr}(j-k,\on{Im}\,x)\times\on{OGr}(m-k,\Sigma),
\eeqn
where $\on{Im}\,x$ is equipped with the non-degenerate bilinear form $(,)$ (see \eqref{bilinear}), and  $\Sigma$ is equipped with the non-degenerate bilinear form  induced by $\langle,\rangle_Q$. It is clear that $\bar V_m\subset\on{OGr}(m-k,\Sigma)$ as $\langle,\rangle_Q|_{V_m}=0$. It is easy to check that $x(V_m^\p)\subset(V_m\cap\on{Im}\,x)^{\p_{(,)}}$ and $\dim x(V_m^\p)=\dim\, (V_m\cap\on{Im}\,x)^{\p_{(,)}}=j-k$. Thus $x(V_m^\p)=(V_m\cap\on{Im}\,x)^{\p_{(,)}}$. Therefore the condition $xV_m^\p\subset V_m$ is equivalent to $(V_m\cap\on{Im}\,x)^{\p_{(,)}}\subset V_m\cap\on{Im}\,x$, i.e. $(V_m\cap\on{Im}\,x)^{\p_{(,)}}\in\on{OGr}(j-k,\on{Im}\,x)$.
This proves the claim.

Thus we obtain a surjective map 
\beq\label{map eta}
\eta:\Theta_{m,j}^{N,k}\to\on{OGr}(j-k,\on{Im}\,x)\times\on{OGr}(m-k,\Sigma)
\eeq 
and it is easy to see that the fibers of $\eta$ are affine spaces $\mathbb{A}^{(j-k)(m-k)}$. Note that the fiber of the natural projection map 
\beq\label{projection}
\Upsilon_{m,j}^{N,k}\to\Theta_{m,j}^{N,k}: (V_{m-1},V_m)\mapsto V_m
\eeq at $V_m$ is the projective space $\bP(V_m/(xV_m^\p))\cong\bP^{m-j+k-1}$.  It is easy to check using the above maps that each piece $\Upsilon_{m,j}^{N,k}$ has an affine paving. Therefore $\Upsilon_{m,j}^{N}$ also has an affine paving.

We can similarly partition $\Gamma_{m,j}^N$  into pieces indexed by the dimension of $V_{m-1}\cap \operatorname{Im}\,x$. Let
\begin{eqnarray*}
&&\Gamma_{m,j}^{N,k}=\{(0\subset V_{m-1}\subset V_m\subset V_m^\p\subset V_{m-1}^\p\subset\bC^{N})\in\Gamma_{m,j}^N\,|\,\dim(V_{m-1}\cap \operatorname{Im}\,x)=k\}\\
&&\Lambda_{m,j}^{N,k}=\{(0\subset V_{m-1}\subset V_{m-1}^\p\subset\bC^N)\,|\,\dim(V_{m-1}\cap \operatorname{Im}\,x)=k,\ xV_{m-1}=0,\ xV_{m-1}^\p\subset V_{m-1}\}.
\end{eqnarray*}
We have a surjective map
\beqn
\eta':\Lambda_{m,j}^{N,k}\to\on{OGr}(j-k,\on{Im}\,x)\times\on{OGr}(m-k-1,\Sigma),\ (V_{m-1})\mapsto((V_{m-1}\cap\on{Im}\,x)^{\p_{(,)}},\bar{V}_{m-1}).
\eeqn
 The fibers of $\eta'$ are affine spaces $\mathbb{A}^{(j-k)(m-k-1)}$. The fiber of the natural projection map $$\Gamma_{m,j}^{N,k}\to\Lambda_{m,j}^{N,k}:\ (V_{m-1},V_m)\mapsto V_{m-1}$$ at a given $V_{m-1}$ is the variety of isotropic lines in $(V_{m-1}^\p\cap\ker\,x)/V_{m-1}$  with respect to the quadratic form induced by $Q$. The same argument as before shows that $\Gamma_{m,j}^{N}$ is paved by affines.

In particular, we see from the above discussion that  
\beq\label{nonempty}
\begin{gathered}
\Upsilon_{m,j}^{N,k}\neq\emptyset\Leftrightarrow\max\{m+j-N/2,j/2,j+1-m\}\leq k\leq\min\{j,m\};\\
\Gamma_{m,j}^{N,k}\neq\emptyset\Leftrightarrow\max\{m+j-N/2-1,j/2,j+1-m\}\leq k\leq\min\{j,m-1\}.
\end{gathered}
\eeq

Finally, in \cite[Proof of Proposition 4.3]{CVX} we have used the following fact 
\begin{lemma}
For $x_i\in\cO_{3^i2^{2m-1-2i}1^{2n-4m+3+i}}$ we have 
\beqn
(\tau_m^{2n+1})^{-1}(x_i)\cong\on{OGr}(m-1-i,2m-1-2i)\,.
\eeqn
\end{lemma}
 This can be deduced from the results in this subsection as follows. Using Lemma \ref{lemma-fiberiso} and \eqref{nonempty} we see that $\tau_m^{2n+1}(x_i)\cong\Upsilon_{m-i,\,2m-1-2i}^{2n+1-3i,\,m-i}$. The conclusion follows by considering the maps in \eqref{map eta} and \eqref{projection}.

\subsection{Families of complete intersections of quadrics and their identification with Hessenberg varieties.}\label{ssec-quadrics}
Let $m\in[1, N-1]$ be an integer. 
For any $s\in \fg_1^{rs}$, let $$X_{m,s}\subset\mathbb P(V)\is\mathbb P^{N-1}$$ be the complete intersection of 
$m$ quadrics $$\langle s^i-,-\rangle_Q=0,\ i=0,\ldots,m-1$$ in $\mathbb P(V)$.
As $s$ varies over $\grs$, we get a family $$\pi_m:X_m\ra\grs$$
of complete intersections of 
$m$ quadrics in $\mathbb P(V)$.

 The families of Hessenberg varieties $\Hess_l^{O,\p}$ and
$\Hess_l^{E,\p}$ over $\Lg_1^{rs}$ are identified with $X_m$'s as follows.
\begin{thm}\label{H=X}
Assume that $k\leq \frac{N-1}{2}$. Then we have 
\begin{enumerate}
\item There is a $K$-equivariant isomorphism $\Hess_k^{O,\p}|_{\grs}\is X_{2k-1}$ of varieties over
$\grs$. 
\item
There is a $K$-equivariant isomorphism $\Hess_k^{E,\p}|_{\grs}\is X_{2k}$ of varieties over
$\grs$.
\end{enumerate}
\end{thm}

We begin with the following simple observation.
\beq\label{ISO}
\text{Let $s\in\fg_1^{rs}$.  
For any isotropic subspace $0\neq U\subset V$,  
$\dim (sU\cap U)< \dim(sU)$.}
\eeq
This follows from the fact that $s$ has no isotropic eigenspaces.

\begin{proof}[Proof of Theorem \ref{H=X}]
We first define a map from $X_{2k-1}$ to $\Hess_n^{O,\p}$.
Let $(s,l)\in X_{2k-1}$, where $s\in \Lg_1^{rs}$ and $l$ is in the 
complete intersection of 
$2k-1$ quadrics $\langle s^i-,-\rangle_Q=0$, $i=0,\ldots,2k-2$, in $\mathbb P(V)$. 
Let $0\neq v\in l$. For $1\leq i\leq k$, consider the subspaces 
$$V_{i}=\on{span}\{ v,sv,...,s^{i-1}v\}.$$ 
Note that $V_{i}$ is isotropic. We show that $\dim V_{i}=i.$ We have $V_{i}=V_{i-1}+ sV_{i-1}$. Thus $\dim V_{i}=\dim V_{i-1}+\dim sV_{i-1}-\dim(sV_{i-1}\cap V_{i-1})>\dim V_{i-1}$, where in the last inequality we use \eqref{ISO}. By induction we see that $\dim V_{i}=i$. Hence
the assignment $(s,l)\mapsto (s,V_{k-1}\subset V_k)$ defines a map 
\[\iota:X_{2k-1}\ra \Hess_k^{O,\p}|_{\Lg_1^{rs}}.\]
One checks readily that $\iota$ is $K$-equivariant. 
We prove that $\iota$ is an isomorphism by constructing an 
explicit inverse map.
Let $(s, V_{k-1}'\subset V_k')\in\Hess_k^{O,\p}$ with $s\in \Lg_1^{rs}$.
We define a sequence of subspaces 
$0\subset V_1'\subset V_2'\subset\cdots\subset V_{k-2}'$ recursively. Let us first define $V_{k-2}'$. Consider the map $\bar{s}:V_{k-1}'\xrightarrow{s} V_k'\to V_k'/V_{k-1}'$. Note that by \eqref{ISO}, the map $\bar s$ is nonzero, hence surjective as $\dim V'_{k}/V'_{k-1}=1$. Let $V_{k-2}'=\ker\bar s$. We have $\dim V_{k-2}'=k-2$ and $V_{k-1}'=V_{k-2}'\cup sV_{k-2}'$.  By induction we can assume that we have defined $V_i'$ such that $\dim V_i'=i$ and 
$V_{i+1}'=V_{i}'\cup sV_{i}'$. Let
$$V_{i-1}'=\ker(\bar s:V_{i}'\xrightarrow{s}{V_{i+1}'}\ra V'_{i+1}/V'_{i}).$$
The same argument as before shows that $\dim V_{i-1}'=i-1$ and 
$V_{i}'=V_{i-1}'\cup sV_{i-1}'$.
Thus in particular we obtain that $\dim V_1'=1$,  
and it is easy to see that the map 
 \[\Hess_k^{O,\p}|_{\Lg_1^{rs}}\ra X_{2k-1}, (s, V_{k-1}'\subset V_k')\mapsto (s,V_1').\]
defines an inverse of $\iota$. 
 This finishes the proof of (1). 
 
 For (2), we observe that, under the isomorphism $\iota: X_{2k-1}\is\Hess_k^{O,\p}$, the equation $\langle s^{2k-1}v,v\rangle_Q=0$ for the divisor 
 $X_{2k}\subset X_{2k-1}$ becomes $\langle sV_{k}',V_k'\rangle_Q=0$,
which is the equation for the divisor $\Hess_k^{E,\p}\subset\Hess_k^{O,\p}$. Thus (2) follows.
\end{proof}

\section{Complete intersections of quadrics
and their double covers}\label{generalization of T}
In \S\ref{ssec-quadrics} we have introduced the families $X_m\ra\grs$ of complete intersections of quadrics, which we have identified with families of Hessenberg varieties $\on{Hess}_n^{E,\p}|_{\Lg_1^{rs}}$, $\on{Hess}_n^{O,\p}|_{\Lg_1^{rs}}$. In order to study the monodromy representations of the equivariant fundamental group $\pi_1^K(\Lg_1^{rs})$ associated with the above families of Hessenberg varieties,  we introduce families $\mY_m$ of branched covers  of $\mathbb P^{N-m-1}$ and 
relate them with  $X_m$. We also introduce a family of branched double covers of $X_m$, denoted by $\tilde{X}_m$, and relate them to families 
 $\tilde\mY_m$
of branched covers 
of $\mathbb P^{N-m-1}$ which we introduce in \S\ref{the case of tX}. Our construction can be regarded as a relative version of the  construction in \cite{T}. 

\subsection{Some notation}

In this section we choose a Cartan subspace $\fa\subset\Lg_1$ that consists of diagonal matrices. Let $\fa^{rs}=\fa\cap\Lg^{rs}$. We write  an element $a\in \fa$ with diagonal entries $a_1,\ldots, a_N$ as $a=(a_1,\ldots,a_N)$ (note that we have diagonalized the elements in $\fa$ with respect to a standard basis $f_i$ of $V$, where $\langle f_i,f_j\rangle_Q=\delta_{i,j}$, rather than the basis chosen in \S\ref{two Hessenberg}). Thus $a=(a_1,\ldots,a_N)\in\ars$ if and only if $a_i\neq a_j$ for $i\neq j$.

Define 
$$I_N:=(\bZ/2\bZ)^{N}/(\bZ/2\bZ)$$ 
where we regard $\bZ/2\bZ$ as a subgroup of $(\bZ/2\bZ)^{N}$ via the diagonal embedding.
For any $\chi\in I_N^\vee=\Hom(I_N,\bG_m)$, we define 
\[\on{supp}(\chi)=\{i\in [1,N]\,|\,\chi(\xi_i)=-1 \}\text{ and }|\chi|=\#\on{supp}(\chi)\]
where $\xi_i$ is the image of $(0,..,1,...,0)\in(\bZ/2\bZ)^{N}$ in $I_N$.
Note that $|\chi|$ is even.

If we identify the centralizer $Z_K(a)$ of $a\in\ars$ with the 
kernel of the map $(\bZ/2\bZ)^{N}\ra \bZ/2\bZ,\ (b_1,...,b_N)\mapsto\sum b_i$, 
we obtain a 
natural map 
\beq\label{natural map}
Z_K(a)\ra (\bZ/2\bZ)^{N}\stackrel{pr}\longrightarrow I_N.
\eeq
Note when $N$ is odd which we will assume from this point on,  the map \eqref{natural map}
is an 
isomorphism. Therefore, in what follows, we often make the canonical identification $Z_K(a)\cong I_N$. 

To emphasize:
\begin{Assumption}
From now on we assume that $N$ is odd.
\end{Assumption}

\subsection{Family of curves}\label{curves}In this subsection we introduce certain families of curves which will be used to construct the families $\mY_m$ and $\tilde\mY_m$ of branched covers of projective spaces.

For any $a=(a_1,...,a_{N})\in \fa^{rs}$, there are 
natural isomorphisms
\[
\pi_1^{ab}(\mathbb P^1-\{a_1,...,a_{N}\})\otimes\bZ/2\bZ\is I_N
\ \ \quad \pi_1^{ab}(\mathbb P^1-\{a_1,...,a_{N},a_{N+1}=\infty\})\otimes\bZ/2\bZ\is I_{N+1}.
\]
The isomorphisms are given by assigning to a small loop around each $a_i$ the element in $I_N$ (resp. $I_{N+1}$) with only non-trivial coordinate in position $i$.
Let $$\text{$C_a\ra\mathbb P^1$ (resp. $\tilde C_a\ra\mathbb P^1$)}$$ be the abelian covering of $\mathbb P^1$ ramified at $\{a_1,...,a_{N}\}$ (resp. $\{a_1,...,a_{N+1}=\infty\}$)
with Galois groups given by $I_N$ (resp. $I_{N+1}$). 
Concretely, $C_a$ (resp. $\tC_a$) is the smooth projective curve corresponding to the 
function field $$\text{$\bC(t)((\frac{t-a_i}{t-a_1})_{i=2,...,N}^{1/2})$ (resp. $\bC(t)((t-a_i)_{i=1,...,N}^{1/2})$).}$$
The group $I_N$ (resp. $I_{N+1}$) acts on $C_a$ (resp. $\tilde C_a$). For any $\chi\in I_N^\vee$ (resp. $\chi\in I_{N+1}^\vee$) we define 
$$
\text{$C_{a,\chi}=C_a/\ker\chi$ (resp. $\tilde C_{a,\chi}=\tilde C_a/\ker\chi$),}
$$
which is a branched double cover of $\mathbb P^1$ with branch locus $\{a_i\,|\,i\in\on{supp}(\chi)\}$. Concretely,  $C_{a,\chi}$ (resp. $\tC_{a,\chi}$) is isomorphic to  
the smooth projective hyperelliptic curve with affine equation 
$$
y^2=\prod_{i\in\on{supp}\chi }(x-a_i)\text{ (resp. }
y^2=\prod_{i\in\on{supp}\chi,\ i\neq N+1 }(x-a_i)).
$$ 
We have $\dim H^1(C_\chi,\bC)=|\chi|-2$ (resp. $\dim H^1(\tilde C_\chi,\bC)=|\chi|-2$).

As $a$ varies over $\ars$, we obtain a family of 
curves $$\text{$C\ra\ars$ (resp. $\tilde C\ra\ars$)}$$ and we similarly obtain families of hyperelliptic curves 
$$\text{$C_\chi\ra\ars$ (resp. $\tilde C_\chi\ra\ars$) for any $\chi$.}$$ 
 We also note that the Weyl group $W=S_N$ acts naturally on $C$ (resp. $\tC$) making the 
projection $C\ra\ars$ (resp. $\tC\ra\ars$) a $W$-equivariant map. 

We will now associate monodromy representations to these families. Let us  fix $a\in\ars$ and a character $\chi\in I_N^\vee$ (resp. $\chi\in I_{N+1}^\vee$) and we recall that $\pi_1(\fa^{rs},a)$ is the pure braid group $P_N$. The monodromy representation of the family  $C_\chi\ra \ars$ factors through the symplectic group and we denote it by 
$\rho_{C_\chi}:P_{N}\ra Sp(H^1(C_{a,\chi},\bC))\is Sp(2m-2)$  where the $m=\frac {|\chi|} 2$. Similarly we obtain a monodromy representation $\rho_{\tilde C_\chi}:P_{N}\ra Sp(H^1(\tilde C_{a,\chi},\bC))\is Sp(2m-2).$

We claim:
\beq
\label{The monodromy representation is Zariski dense}
\begin{gathered}
\text{The images of the representations $\rho_{C_\chi}$ and $\rho_{\tilde C_\chi}$  are Zariski dense}
\\
\text{in $Sp(H^1(C_{a,\chi},\bC))$ and  $Sp(H^1(\tilde C_{a,\chi},\bC))$, respectively}\,.
\\
\text{In particular, the representations $\rho_{C_\chi}$ and $\rho_{\tilde C_\chi}$ are irreducible}\,.
\end{gathered}
\eeq
We see this as follows. Consider the following subvariety $\fa^{rs}(a,\chi)\subset\fa^{rs}$
\beq\label{subvariety of ars}
\fa^{rs}(a,\chi)=\{a'\in\fa^{rs}\,|\,a'_i=a_i\ \text{if}\ i\notin\on{supp}(\chi) \}.
\eeq
It suffices to show that the monodromy representation of the restriction of 
$C_\chi\ra\ars$ to $\fa^{rs}(a,\chi)$ has Zariski dense image. Now, $\fa^{rs}(a,\chi)$ is an open subset of the space $\cM_{2m}$ of $2m$ distinct marked points in $\bC$ and the family $C_\chi\times_\ars\fa^{rs}(a,\chi)$ is the restriction of the universal family of hyperelliptic curves parametrized by $\cM_{2m}$. Note further, that $\cM_{2m}$ itself is an open subset of the space $\tilde\cM_{2m}$ of $2m$ distinct marked points in $\bP^1$ carrying its own family of hyperelliptic curves.  Now, by \cite{A} (see also \cite[Theorem 10.1.18.3]{KS}), the monodromy representation on $\tilde\cM_{2m}$ is irreducible and has Zariski dense image. Therefore 
$\rho_{C_\chi}$, as a restriction to an open subset, has the same property. The argument in the case $\tilde C_\chi\ra\ars$ is completely analogous except one has to take into account that $a_{N+1}=\infty$.

Finally, there is a unique character $\chi_0\in I_{N+1}^\vee$ with 
$|\chi_0|=N+1$ (here we use the assumption that $N$ is odd). The character $\chi_0$ is invariant under the Weyl group action. Thus we can pass to a quotient of $\tC_{\chi_0}\ra\ars$ under the $W$ action and in this way obtain a family $\overline C_{\chi_0}\ra\fc^{rs}=\ars/W$. The family $\overline C_{\chi_0}\ra\fc^{rs}=\ars/W$ is the universal family 
of hyperelliptic curves $y^2=\prod_{i=1}^N(x-a_i)$  and $\tC_{\chi_0}\ra\ars$ is a similar universal family  with marked ramification points.

\subsection{Branched cover $\mY_m$ of projective spaces and $X_m$}
\label{branched cover X_m}
Define 
\beqn\bar I^{N-m-1}_N=\ker(sum:I_N^{N-m-1}\ra I_N),
\eeqn
where $sum$ is the summation map. Fix $a=(a_1,\ldots,a_N)\in\ars$. Let $C_a\ra\bbP^1$ be the 
curve introduced 
in \S\ref{curves}.
The semi-direct  product $\bar I_N^{N-m-1}\rtimes S_{N-m-1}$ acts naturally
on $C_a^{N-m-1}$ and we define $$\mY_{m,a}=C_a^{N-m-1}/\bar I_N^{N-m-1}\rtimes S_{N-m-1}.$$
We have a natural map 
\[\iota_a:\mY_{m,a}\ra C_a^{N-m-1}/I^{N-m-1}_N\rtimes S_{N-m-1}\is\mathbb P^{N-m-1}.\] 
According to \cite[Proposition 2.4.4]{T}, 
for a suitable choice of homogeneous coordinates\linebreak $[x_1,...,x_{N-m}]$
of $\bbP^{N-m-1}$,
each ramification point $a_i$
defines a 
hyperplane  
\beq\label{H_a}
H_{a,i}=x_1+a_ix_2+\cdot\cdot +a^{N-m-1}_ix_{N-m}=0
\eeq
in $\bbP^{N-m-1}$ and the map $\iota_a$
is an $I_N$-branched cover of $\mathbb P^{N-m-1}$ with branch locus 
$\{H_{a,i}=0\}_{i=1,...,N}$. As $a$ varies over $\ars$, we get a $\ars$-family
of $I_N$-branched 
covers of $\mathbb P^{N-m-1}$
\beqn
\xymatrix{\mY_m
\ar[rd]\ar[rr]^{ \iota}&&\mathbb P^{N-m-1}_\ars\ar[ld]
\\&\ars&}
\eeqn
where $$\mY_m=
C^{N-m-1}/\bar I_N^{N-m-1}\rtimes S_{N-m-1}$$ and the 
base change of $\iota$ to $a$ is equal to $\iota_a$.
Observe that the $ W$-action on $C$ induces a $ W$-action on $\mY_m$
making the projection $\mY_m\ra\ars$ a $W$-equivariant map.

Let $X_{m}\ra\grs$ be the family of complete intersections of quadrics introduced in \S\ref{ssec-quadrics}.
For $a=(a_1,...,a_{N})\in\ars$ the 
equation of $X_{m,a}$ is given by 
\[a_1^iv_1^2+\cdot\cdot+a_{N}^iv_{N}^2=0,\ \ \ i=0,...,m-1.\]
Consider the map $$s:\mathbb P(V)\ra\mathbb P(V),\  
[v_1,...,v_{N}]\mapsto [v_1^2,...,v_{N}^2].$$
The image 
$s(X_{m,a})$ is equal to $\mathbb P(V_{m,a})$, where 
\beq\label{V_a}
V_{m,a}=\{v\in V\,|\, a_1^iv_1+\cdot\cdot+a_{N}^iv_{N}=0,\ i=0,...,m-1\}\subset V.
\eeq  
The resulting 
map \[s_{a}:X_{m,a}\ra\mathbb P(V_{m,a})\]
is an $I_N$-branched cover with branch locus $\{v_i=0\}_{i=1,...,N}$. 
As $a$ varies over $\ars$, 
we obtain 
\beqn
\xymatrix{X_m|_\ars\ar[rd]\ar[rr]^s&&\mathbb P(V_m)\ar[ld]
\\&\ars&}
\eeqn
Here $V_m\ra\ars$ is the vector bundle over $\ars$ whose fiber over 
$a$ is $V_{m,a}$ and $\mathbb P(V_m)$ is the associated projective bundle.

The two families 
$X_m|_\ars$ and $\mY_m$ are related as follows. Let 
\[
(\tilde\fa^{rs})'=\{(a,c)\,|\,a=(a_1,...,a_N)\in\ars,
c=(c_1,...,c_N)\in\bC^N,
\ c_i^2=d_i:=\prod_{j\neq i}(a_j-a_i)\}.
\]
The projection $(\tilde\fa^{rs})'\ra \ars,\ (a_1,...,a_N,c_1,...,c_N)\mapsto (a_1,...,a_N)$ 
realizes $(\tilde\fa^{rs})'$ as a 
$(\bZ/2\bZ)^N$-torsor over $\ars$. 
Consider the following $I_N$-torsor over $\ars$
\beq\label{tars}
\tilde\fa^{rs}:=(\tilde\fa^{rs})'/(\bZ/2\bZ)\ra\ars;
\eeq
here  we view $\bZ/2\bZ$ as a subgroup of  $(\bZ/2\bZ)^N$  via the diagonal embedding.  The Weyl group $W$ acts naturally on 
$\tilde\fa^{rs}$ making the projection to $\ars$ a $W$-equivariant map.  We observe that the $I_N$-torsor $\tilde{\fa}^{rs}$ of \eqref{tars} gives rise to the following canonical map
\beq\label{the map rho}
\rho:\pi_1(\ars,a)\cong P_N\to I_N.
\eeq

\begin{proposition}\label{prop-xy}
We have an $I_N\rtimes{W}$-equivariant isomorphism 
\beq\label{X_m=Y_m twist}
X_m|_\ars\is(\mY_m\times_\ars\tilde\fa^{rs})/I_N:=\mY_m^t,
\eeq
where ${W}$ (resp. $I_{N}$) acts on $\mY_m^t$ by the diagonal action (resp. 
on the first factor).
\end{proposition}
\begin{proof}
Following \cite[\S 5]{T}, we 
consider the family 
\[X_m'|_\ars\ra\ars\]
whose fiber over $a=(a_1,\ldots,a_N)\in\ars$ is the complete intersection of $m$ quadrics in $\mathbb P(V)$
given by
\[\frac{a_1^i}{d_1}v_1^2+\cdot\cdot+\frac{a_{N}^i}{d_N}v_{N}^2=0,\ \ \ i=0,...,m-1,\]
where $d_i:=\prod_{j\neq i}(a_j-a_i)$. One can think of $X_{m}|_\ars$ as 
a \emph{twist} of $X_{m}'|_\ars$. More precisely,  
we have a natural map $$\tilde\fa^{rs}\times_{\ars} X'_{m}|_\ars\ra X_m|_\ars,\  (a,c,[v_1,...,v_N])\mapsto(a,[v_1/c_1,\ldots,v_N/c_N])$$ and it is not hard to see that 
it descends to  
a canonical $I_N$-equivariant isomorphism 
\beq\label{twist}
X_m|_{\ars}\is(\tilde\fa^{rs}\times_{\ars} X_{m}'|_\ars)/I_N.
\eeq
Here $I_N$ acts on the product via the diagonal action.
The Weyl group ${W}=S_N$ acts naturally on $X_m|_{\ars}$, $X_m'|_\ars$, 
$\tilde\fa^{rs}$, and $(\tilde\fa^{rs}\times_{\ars} X_{m}'|_\ars)/I_N$, making the projections to $\ars$ equivariant maps under the $W$-actions. Moreover, the isomorphism in \eqref{twist} is also $W$-equivariant. In \S\ref{proof of twist=tY_m} (see Proposition \ref{twist=tY_m}),
we show that 
there is an $I_N\rtimes W$-equivariant isomorphism 
\beq\label{X_m'=Y_m}
X_m'|_\ars\is\mY_m.
\eeq 
Combining \eqref{twist} with \eqref{X_m'=Y_m}
we obtain \eqref{X_m=Y_m twist}.
\end{proof}

\subsection{Branched double covers $\tilde{X}_m$ of complete intersections of quadrics}\label{quadrics}

We introduce a branched double cover of $X_m$ as follows.
Let $\widetilde{V}=V\oplus\bC$. 
For any $s\in\grs$, consider the following quadrics in $\mathbb P(\tV)$,
\beqn
\begin{gathered}
\tQ_{i}(v, \epsilon)=\langle s^iv,v\rangle_Q=0,\ i=0,\ldots,m-1\\
\tQ_{m}(v, \epsilon)=
\langle s^mv,v\rangle_Q-\epsilon^2=0.
\end{gathered} 
\eeqn
We define 
$\tX_{m,s}$ to be the complete intersection of 
$m+1$ quadrics $\tQ_{i}=0$, $i=0,\ldots,m$. 
As $s$ varies over $\grs$, we get a family $$\tilde\pi_m:\tX_m\ra\grs$$
of complete intersections of $m+1$ quadrics in $\bbP(\tV)$.
The projection $\tV=V\oplus\bC\ra V$,
$(v, \epsilon)\mapsto v$, induces a map
$p_{m}:\tX_{m}\ra X_{m}$
which is 
a branched double cover with branch locus $X_{m+1}\subset X_{m}$. 

The map $\tV=V\oplus\bC\to\tV$ given by $(v, \epsilon)\mapsto (v, -\epsilon)$ defines an involution on $\tX_m$.
We denote this involution by $\sigma$.

The group $K=SO(V,Q)$ acts naturally on both $X_m$ and $\tX_m$. The maps 
$\pi_m:X_m\ra\grs\text{ and }\tilde\pi_m:\tX_m\ra\grs$ are $K$-equivariant.
In particular, the centralizer $Z_{K}(s)$ acts on the fibers $X_{m,s}$ and $\tX_{m,s}$.

\subsection{Branched cover $\tilde{\mY}_m$ of projective spaces and $\tX_m$}\label{the case of tX}
In this subsection
we generalize Proposition \ref{prop-xy} to the  branched double cover 
$\tX_m$ introduced in \S\ref{quadrics}. 

For $a=(a_1,...,a_{N})\in\ars$ the 
equations of $\tX_{m,a}\subset\bbP^{N-m-1}(\tV)$ (recall that $\tV=V\oplus\bC$)
are given by 
\[a_1^iv_1^2+\cdot\cdot+a_{N}^iv_{N}^2=0,\ \ i=0,...,m-1,\ \ 
a_1^mv_1^2+\cdot\cdot+a_{N}^mv_{N}^2-\epsilon^2=0
.\]
Consider the map 
\beq\label{the map ts}
\tilde s:\mathbb P(\tV)\ra\mathbb P(\tV),\  
[v_1,...,v_{N},\epsilon=v_{N+1}]\mapsto [v_1^2,...,v_{N}^2,v_{N+1}^2].
\eeq
We have 
$\tilde s(\tX_{m,a})\is\mathbb P(\tV_{m,a})$,
where $\tV_{m,a}\subset \tV$ is the subspace defined by 
the equations
\beqn\label{W_a}
 a_1^iv_1+\cdot\cdot+a_{N}^iv_{N}=0,\ \ i=0,...,m-1,\ \ 
a_1^mv_1+\cdot\cdot+a_{N}^mv_{N}-v_{N+1}=0.
\eeqn
The
map 
\beqn\label{s_a}
\tilde s_{a}:\tX_{m,a}\ra\mathbb P(\tV_{m,a})
\eeqn
is a $I_{N+1}$-branched cover with branch locus $\{v_i=0\}_{i=1,...,N+1}$.
As $a$ varies over $\ars$, 
we obtain  
\beqn
\xymatrix{\tX_m|_\ars\ar[dr]\ar[rr]^{\tilde s\ }&&\mathbb P(\tV_m)\ar[dl]
\\&\ars&}
\eeqn
here $\tV_m$ is the vector subbundle 
of the trivial bundle $\tV\times\ars$
whose fiber 
over $a$ is the subspace $\tV_{m,a}$, and $\bbP(\tV_m)$ is the associated projective bundle.

We now introduce another family $\tilde{\mY}_m$ of branched covers of $\bP^{N-m-1}$. Let $sum:I_{N+1}^{N-m-1}\ra I_{N+1}$ be the summation map
and define $\bar I^{N-m-1}_{N+1}=\ker(sum)$. For any $a\in \ars$ let 
$\tilde C_a\ra\mathbb P^1$ be the $I_{N+1}$-branched cover of $\mathbb P^1$
introduced in \S\ref{curves}.
The semi-direct  product $\bar I_{N+1}^{N-m-1}\rtimes S_{N-m-1}$ acts naturally
on $(\tilde C_a)^{N-m-1}$. We define $$\tilde\mY_{m,a}=(\tilde C_a)^{N-m-1}/\bar I_{N+1}^{N-m-1}\rtimes S_{N-m-1}.$$
Similar to the case of $\mY_{m,a}$, the natural map 
\[\tilde\iota_a:\tilde\mY_{m,a}\ra (\tilde C_a)^{N-m-1}/I^{N-m-1}_{N+1}\rtimes S_{N-m-1}\is\mathbb P^{N-m-1}\] 
is an $I_{N+1}$-branched cover of $\mathbb P^{N-m-1}$ with branch locus 
$\{H_{a,i}=0\}_{i=1,...,N+1}$, here 
$H_{a,i}=0$ for $i=1,...,N$ are the hyperplanes as before (see \eqref{H_a}) and 
$H_{a,N+1}:=x_{N-m}=0$ is the hyperplane corresponding to
the ramification point $a_{N+1}=\infty$.

As $a$ varies over $\ars$, we get an $\ars$-family of $I_{N+1}$-branched 
cover of $\mathbb P^{N-m-1}$
\[
\xymatrix{\tilde\mY_m\ar[dr]\ar[rr]^{\tilde\iota\ }&&\mathbb P^{N-m-1}_\ars\ar[dl]
\\&\ars&}
\]

We will again make us of $\tilde\fa^{rs}$ of
\eqref{tars} to relate the two families 
$\widetilde{X}_m|_\ars$ and $\tilde{\mY}_m$. The Weyl group $ W$ acts naturally on $\tX_m|_\ars$ and $\tilde\mY_m$,  
making the projections to $\ars$  equivariant with respect to these $W$-actions. We let $I_N$ act on $\tilde\mY_m$ via the map 
\beq\label{embedding}
\kappa:I_N\cong Z_K(a)\hookrightarrow(\bZ/2\bZ)^{N+1}\xrightarrow{pr} I_{N+1},
\eeq
where the first arrow  is given by $(\zeta_1,\ldots,\zeta_N)\mapsto(\zeta_1,\ldots,\zeta_N,0)$. 
We also let $W$ act on $I_{N+1}$ by permuting the first $N$ 
coordinates and we use this convention to form the semi-direct product $I_{N+1}\rtimes W$.

\begin{proposition}
There is an 
$I_{N+1}\rtimes W$-equivaraint 
isomorphism 
\beq\label{tX_m=tY_m}
\tX_m|_\ars\is\tilde\mY_m^t:=(\tilde\mY_m\times_\ars\tilde\fa^{rs})/I_N,
\eeq
where ${W}$ (resp. $I_{N+1}$) acts on $\tilde\mY_m^t$ by the diagonal action (resp. 
on the first factor).
\end{proposition}

\begin{proof}
Let us consider the following 
\emph{twist} of $\tX_m|_\ars$:
\[\tX_m'|_\ars\ra\ars\]
whose fiber over $a=(a_1,\ldots,a_N)$ is the complete intersection of quadrics 
given by 
\[\frac{a_1^i}{d_1}v_1^2+\cdot\cdot+\frac{a_{N}^i}{d_N}v_{N}^2=0,\ \ \ i=0,...,m-1,
\ \ 
\frac{a_1^m}{d_1}v_1^2+\cdot\cdot+\frac{a_{N}^m}{d_N}v_{N}^2-\epsilon^2=0
\]
where $d_i$ is defined as before, i.e., $d_i=\prod_{j\neq i}(a_j-a_i)$. Similar to the case of $X_m$, we have a canonical isomorphism 
\beq\label{twist tX_m}
\tX_m|_{\ars}\is(\tilde\fa^{rs}\times_\ars\tX_m'|_\ars)/I_N.
\eeq
 Here the $I_N$-action on $\tX_m'|_\ars$ is defined as the composition of 
$I_N\ra I_{N+1}$ in \eqref{embedding} with
the natural action of $I_{N+1}$ on $\tX_m$. The Weyl group $ W$ acts naturally on $\tX_m$, $\tX_m'$ and $\tilde\mY_m$,  
making the projections to $\ars$   equivariant maps under the $W$-actions. Thus we obtain 
$I_{N+1}\rtimes W$-actions 
on $\tX_m$, $\tX_m'$ and $\tilde\mY_m$ and the projections to
$\ars$ are $I_{N+1}\rtimes W$-equivariant.   Moreover, the isomorphism in \eqref{twist tX_m} is $I_{N+1}\rtimes W$-equivariant.
In \S\ref{proof of twist=tY_m} (see Proposition \ref{twist=tY_m}), we show that there is an $I_{N+1}\rtimes W$-equivariant isomorphism $\tX_{m}'|_{\ars}\is\tilde\mY_{m}$. Combining this with \eqref{twist tX_m} we obtain \eqref{tX_m=tY_m}.
\end{proof}

\subsection{The families $\tX_m'$ and $\tilde\mY_m$}\label{proof of twist=tY_m} In this subsection we state and prove  the following proposition which was used in the previous subsections.
\begin{prop}\label{twist=tY_m}
We have an $I_{N+1}\rtimes W$-equivariant isomorphism $\tX_{m}'|_{\ars}\is\tilde\mY_{m}$.
In particular, it induces an $I_N\rtimes W$-equivariant isomorphism on the 
quotient \[X_m'|_\ars\is\tX_m'|_\ars/(\bZ/2\bZ)\is\tilde\mY_m/(\bZ/2\bZ)\is\mY_m.\]
Here $\bZ/2\bZ$ acts on $\tX_m'$ and $\tilde\mY_m$ via the map
$\bZ/2\bZ\ra I_{N+1}$ given by $1\mapsto {(0,...,0,1)}$. 
\end{prop}
We follow closely the argument in \cite[\S 2]{T}.
We begin by introducing some auxiliary spaces and maps.
Let $\tV_m'\subset \tV\times\ars$ be the vector sub-bundle whose fiber 
over $a\in\ars$ is the subspace 
$\tV_{m,a}'\subset \tV$
defined by the equations
\beq\label{eqns for tV'}
 \frac{a_1^i}{d_1}v_1+\cdot\cdot+\frac{a_{N}^i}{d_N}v_{N}=0,\ \ i=0,...,m-1,\ \ 
 \frac{a_1^m}{d_1}v_1+\cdot\cdot+\frac{a_{N}^m}{d_N}v_{N}-v_{N+1}=0.
\eeq
The map $\tilde{s}:\mathbb P(\tV)\times\ars\ra\mathbb P(\tV)\times \ars$ (see \eqref{the map ts})
maps $\tX_m'|_\ars$ to $\mathbb P(\tV_m')$ and the resulting map 
\[\tilde s':\tX_m'|_\ars\ra\mathbb P(\tV_m')\]
is an $I_{N+1}$-branched cover with branch locus $\{v_i=0,\ i=1,...,N+1\}$.

Let $\bC(\mathbb P(\tV_m'))$ be the function field of $\mathbb P(\tV_m')$.
Then the function field of $\tX_m'|_\ars$ is given by the following field extension 
$$F:=\bC(\mathbb P(\tV_m'))((\frac{v_i}{v_{N+1}})^{1/2}_{i=1,...,N})\supset 
\bC(\mathbb P(\tV_m')).$$ Since $\tX_m'|_\ars$ is smooth and $\tilde s'$ is 
finite, it follows that 
\beqn
\text{$\tX_m'|_\ars$ is the normalization of $\mathbb P(\tV_m')$
in $F$.}\footnote{Recall for any irreducible variety $X$ and $K$ a finite extension of the 
function field $\bC(X)$, there exists a unique normal variety $Y$ and a finite morphism 
$f:Y\ra X$ such that the induced map $\bC(X)\ra\bC(Y)=K$ is the given field extension.
We call $Y$ the \emph{normalization} of $X$ in $K$.}
\eeqn  The group $I_{N+1}$ acts on $F$ by  
$\displaystyle{\zeta:\,v_i^{1/2}\mapsto(-1)^{\zeta_i}\,v_i^{1/2},\ \zeta=(\zeta_1,...,\zeta_{N+1})\in I_{N+1}}$ and 
$W$ acts on $F$ by  $\displaystyle{w:\  (\frac{v_i}{v_{N+1}})^{1/2}\mapsto(\frac{v_{w(i)}}{v_{N+1}})^{1/2}}$.

Similarly, let 
$k_\eta$ be the function field of $\ars$ and let 
$\eta=(a_1,...,a_N)\in\ars(k_\eta)$ be the corresponding generic point.
Then the function field of $\tilde\mY_m$ is given by the following field extension 
$$F'=\bC(\mathbb P^{N-m-1}_\ars)((\frac{H_{\eta,i}}{H_{\eta,N+1}})^{1/2}_{i=1,...,N})\supset\bC(\mathbb P^{N-m-1}_\ars).$$ Here,
$H_{\eta,i},\, i=1,...,N$ are the hyperplanes associated to $\eta\in\ars$ in \eqref{H_a}, $H_{\eta,N+1}=x_{N-m}$, and 
$\frac{H_{\eta,i}}{H_{\eta,N+1}}$ are rational functions on 
$\mathbb P^{N-m-1}_\eta$, regraded as elements in 
$\bC(\mathbb P^{N-m-1}_\eta)=\bC(\mathbb P^{N-m-1}_\ars)$.
Since $\tilde \mY_m$ is smooth and $\tilde\iota:\tilde\mY_m\ra\mathbb P_\ars^{N-m-1}$ is finite, it follows that 
\beqn
\text{$\tilde \mY_m$ is the normalization of $
\mathbb P^{N-m-1}_\ars$ in $F'$.}
\eeqn
The group $I_{N+1}$ acts on $F'$ by 
$\displaystyle{\zeta:\,H_{\eta,i}^{1/2}\mapsto(-1)^{\zeta_i}\,H_{\eta,i}^{1/2},\ \zeta=(\zeta_1,...,\zeta_{N+1})\in I_{N+1}}$ and 
$ W$ acts on $F$ by  $\displaystyle{w:\,(\frac{H_{\eta,i}}{H_{\eta,N+1}})^{1/2}\mapsto(\frac{H_{\eta,w(i)}}{H_{\eta,N+1}})^{1/2}}$.

By the discussion above, to prove Proposition \ref{twist=tY_m},
it is enough to prove the following statement. 
\beq\label{configurations}
\begin{gathered}
\text{The two $\ars$-families of configurations}\\ 
(\mathbb P(\tV'_{m,a}),\{v_i\}_{i=1,...N+1})\ \text{and}\ \ 
(\mathbb P^{N-m-1}_{\ars},\{H_{a,i}\}_{i=1,...
,N+1})\text{ are 
equivalent. }
\end{gathered}
\eeq
That is, 
there is an isomorphism (or trivialization) of
vector bundles $\phi:\bC^{N-m}\times\ars\is \tV_m'$ over $\ars$
such that for any $k$-point $a\in\ars(k)$, $k$ a field,
the induced map on the dual fibers $\phi^*_a:(\tV_{m,a}')^*\is k^{N-m}$
satisfies $\phi^*_a(v_i)=H_{a,i}$ for $i=1,...,N+1$.

To prove \eqref{configurations}, we need to construct, 
for each $S$-point $a\in\ars(S)$, 
a functorial isomorphism
$\phi_a:\bC^{N-m}\times S\is \tV_{m,a}'$ satisfying the desired property. 
For notational simplicity we construct such isomorphisms on the level of $k$-points. The argument for general $S$-points is the same.

Consider the following map
\[\psi_a:\tV\otimes k\stackrel{pr}\ra V\otimes k\stackrel{\times d}\is V\otimes k\]
where $pr:\tV\otimes k=(V\otimes k)\oplus k\ra V\otimes k$
is the projection map and the second isomorphism is given by multiplying
the diagonal matrix $d=\on{diag}(d_1^{-1},...,d_N^{-1})\in GL(V\otimes k)$ 
(recall for $a=(a_1,...,a_N)\in\ars$, $d_i=\prod_{j\neq i} (a_j-a_i)$). One can check that $\psi_a$ maps $\tV_{m,a}'$ isomorphically onto $V_{m,a}$
(see \eqref{eqns for tV'} and \eqref{V_a} for the definitions of $\tV_{m,a}'$ and $V_{m,a}$, respectively) and the resulting isomorphism 
$\psi_a:\tV_{m,a}'\is V_{m,a}$ 
satisfies 
$$\text{$\psi_a^*(d_i\cdot v_i)=v_i$ for $i=1,...,N$ and 
$\psi_a^*(H_\infty)=v_{N+1}$,}$$ where $H_\infty:=a_1^mv_1+\cdot\cdot\cdot+a_N^mv_N$. Thus we are reduced to show that 
$$\text{$(\bbP(V_{m,a}), \{d_i\cdot v_i\}_{i=1,...,N}\cup H_\infty)$ and 
$(\mathbb P(k^{N-m}), \{H_{a,i}\}_{i=1,...N+1})$ are equivalent,}$$ that is,
there is an isomorphism \[\gamma_a:k^{N-m}\is V_{m,a}\]
such that $\gamma_a^*(d_i\cdot v_i)=H_{a,i}$, $i=1,\ldots,N$ and $\gamma_a^*(H_\infty)=H_{a,N+1}$.

Consider the basis $u_i=(a_1^{i},...,a_N^{i})$, $i=0,...,N-1$
of $V\otimes k$. Then the isomorphism $V\otimes k\is (V\otimes k)^*$, given by the paring 
$\langle(v_i),(w_i)\rangle=\sum v_iw_i$, induces the following isomorphism  
\[f_1:V_{m,a}^*\is V\otimes k/k\langle u_0,...,u_{m-1}\rangle\is k\langle u_{m},...,u_{N-1}\rangle.\]
Let $s_{i}$ be the elementary symmetric polynomial in $a_1,...,a_N$
of degree $i$ and let
$A=(a_{ij})\in GL_{N-m}(k)$ be the matrix with 
entries $a_{ij}=(-1)^{i-1}s_{j-i}$ if $j\geq i$ and $a_{ij}=0$ otherwise.
Consider the following isomorphism
\[f:V_{m,a}^*\stackrel{f_1}\is k\langle u_{m},...,u_{N-1}\rangle\stackrel{f_2}\is k^{N-m}\stackrel{f_3}\is k^{N-m},\]
here $f_2:k\langle u_{m},...,u_{N-1}\rangle\is k^{N-m}$ is 
the isomorphism given by $u_{N-i}\mapsto(-1)^{i-1}x_{i}$\footnote{Here 
we regard $x_i$ as the $i$-th coordinate vector of $k^{N-m}$.} and 
$f_3$ is the isomorphism given by right multiplication by $A^{-1}$. 
We claim that the dual $$\gamma_a:=f^*:k^{N-m}\is V_{m,a}$$
is the desired isomorphism, i.e., we have 
$f(d_i\cdot v_i)=H_{a,i}$ for $i=1,...,N$, and $f(H_\infty)=H_{a,N+1}$.
Note that  the configuration $(\bbP(V_{m,a}), \{d_i\cdot v_i\}_{i=1,...,N})$
(resp. $(\mathbb P(k^{N-m}), \{H_{a,i}\}_{i=1,...N})$)
is equal to the configuration $(P, H_1,...,H_N)$ (resp. $(P',H_1',...,H_N')$)
in \cite[\S 2.1]{T}. Moreover, the map $f$ is the one used in  \cite[Proof of Theorem 2.1.1]{T} to show that
the above configurations are equivalent. Thus   
according to \cite[Proof of Theorem 2.1.1]{T} we have
\[f(d_i\cdot v_i)=H_{a,i}\ \ \text{for}\ \ i=1,...,N.\] 
So it remains to show that $f(H_\infty)=H_{a,N+1}$.
For this we 
observe that $f_1(H_\infty)=u_m$. Hence 
\beqn
f(H_\infty)=f_3\circ f_2\circ f_1(H_\infty)=f_3\circ f_2(u_m)
=(-1)^{N-m-1}\cdot f_3(x_{N-m})
\eeqn
\beqn
=(-1)^{N-m-1}\cdot x_{N-m}\cdot A
=x_{N-m}=H_{a,N+1}.
\eeqn
This proves \eqref{configurations}.

The proof of Proposition \ref{twist=tY_m} is complete.

\section{Monodromy of families of Hessenberg varieties}\label{sec-monodromy of Hess}
In this section we study the monodromy 
representation of $\pi_1^K(\grs,a)=Z_K(a)\rtimes B_N$ on the primitive cohomology of 
complete intersections of quadrics $X_m$ (and on the primitive cohomology of their branched double covers $\tX_m$). By Theorem \ref{H=X} this gives us a 
complete description of the monodromy representations of $\pi_1^K(\Lg_1^{rs})$ associated to the families of Hessenberg varieties $\Hess_n^{O,\p}$ and 
$\Hess_n^{E,\p}$.

To state the result, let us recall, from~\S\ref{curves}, the monodromy representations $\rho_{C_\chi}:P_{N}\ra Sp(H^1(C_{a,\chi},\bC))\is Sp(2i-2)$  and  $\rho_{\tilde C_\chi}:P_{N}\ra Sp(H^1(\tilde C_{a,\chi},\bC))\is Sp(2i-2)$ where  $i=\frac {|\chi|} 2$. Recall further that, by~\eqref{The monodromy representation is Zariski dense}, these representations are irreducible with Zariski dense image. Let us consider the irreducible representation of $Sp(2i-2)$ associated to the fundamental weight $\omega_j$. Composing $\rho_{C_\chi}$  and  $\rho_{\tilde C_\chi}$ with this fundamental representation we obtain irreducible representations  $\on{P}^j_\chi$ and $\tilde{\on{P}}^j_\chi$ of the pure braid group $P_N$. 

For a character $\chi$ of an abelian group we write $V_\chi$ for the corresponding one dimensional representation. Recall that the group $Z_K(a)$ can be naturally identified with $I_N$ as explained in~\eqref{natural map}. We also relate the characters of $I_{N}$ and $I_{N+1}$ using the map $\kappa$ defined in \eqref{embedding}. 
From these considerations we conclude that 
\beq
\label{characters}
Z_K(a)^\vee=I_N^\vee \qquad \text{and we have a map}  \qquad    \check{\kappa}: I_{N+1}^\vee    \ra     I_N^\vee\,.
\eeq
In particular, characters of $I_N$ and $I_{N+1}$ can be regarded as characters of $Z_K(a)$.
To state the main theorems of this section we define two  $Z_K(a)\rtimes P_N$-representations as follows:
\[E_{ij}^{N}\is
\bigoplus_{\chi\in I_{N}^\vee,\,|\chi|=2i}\on{P}^j_\chi\otimes V_\chi\qquad\text{and}\qquad \tE_{ij}^N=
\bigoplus_{\substack{\chi\in I_{N+1}^\vee,\,|\chi|=2i,\\N+1\in\on{supp}\chi}}\tilde{\on{P}}^j_\chi\otimes V_\chi,\]
where the $I_N$ acts on $\tE_{ij}^N$ via the map $\check{\kappa}:I_{N+1}^\vee    \ra     I_N^\vee$ of~\eqref{characters}, and  $P_N$ acts on $V_\chi$ via the map $\rho:P_N\to I_N$ of \eqref{the map rho}.
Lemmas~\ref{E_ij} and~\ref{tE_ij}  show that the $Z_K(a)\rtimes P_N$ actions on $E_{ij}^{N}$ and $\tE_{ij}^N$ extend naturally to  $Z_K(a)\rtimes B_N$-actions.

The main results of this section are the following.
\begin{thm}\label{irred}

For $1\leq m\leq N-1$, the monodromy representation of $\pi_1^K(\grs,a)$ on $P(X_m):=H^{N-m-1}_{prim}(X_{m,a},\bC)$ 
decomposes into irreducible representations in the following manner:
\[P(X_m)\is
\bigoplus_{i}^{}\bigoplus_{\substack{j\equiv N-m-1\on{mod}2\\ j\in [0,l]}}E_{ij}^N,\]
with $N-m+1\leq 2i\leq N$,   
$l=\min\{N-m-1,-N+m+2i-1\}$.
\end{thm}

To state the second main result, we set
$$  
P(\tX_m):=H^{N-m-1}_{prim}(\tX_{m,a},\bC).
$$ 
Recall that there is an involution action $\sigma$ on $\tilde X_{m}$ 
and the projection map $p_m:\tilde X_m\ra X_m$ is a branched double cover
with Galois group $\langle\sigma\rangle\is\bZ/2\bZ$
(see  \S\ref{quadrics}). Then $P(\tX_m)= P(\tX_m)^{\sigma=id}\oplus P(\tX_m)^{\sigma=-id}$ and we have
 $P(\tX_m)^{\sigma=id}=P(X_m)$. The next theorem describes $P(\tX_m)^{\sigma=-id}$. 
 
\begin{thm}\label{irred tE}

For $1\leq m\leq N-1$, 
the monodromy representation 
of $\pi_1^K(\grs,a)$ on 
$P(\tX_m)^{\sigma=-id}$ decomposes into irreducible representations in the following manner:
\[P(\tX_m)^{\sigma=-id}\is
\bigoplus_{i}^{}\bigoplus_{\substack{j\equiv N-m-1\on{mod}2\\ j\in[0,l]}}^{}\tE_{ij}^N,\]
with $N-m+1\leq 2i\leq N+1$, 
$l=\min\{N-m-1,-N+m+2i-1\}$.
\end{thm}

\subsection{Proof of Theorem \ref{irred}}
Let us start with the following proposition which is a consequence of Proposition \ref{prop-xy}.
\begin{prop}\label{mon thm X_m}
There is an isomorphism 
of representations of $\pi_1^{K}(\grs,a)\is
I_N\rtimes B_N$
\beqn
H^i(X_{m,a},\bC)\is\bigoplus_{\chi\in I_N^\vee}H^i(\mY_{m,a},\bC)_{\chi}\otimes V_\chi.
\eeqn
The group $I_N$ acts on the summand $H^i(\mY_{m,a},\bC)_\chi\otimes V_\chi$
via the character $\chi\in I_N^\vee$.
\end{prop}
\begin{proof}
Observe that the families 
$X_m|_\ars\ra\ars$,
$\mY_m\ra\ars$, and $\tilde\fa^{rs}\ra\ars$ are all 
$W$-equivariant. Hence 
their cohomology groups $H^i(X_{m,a},\bC)$, $H^i(\mY_{m,a},\bC)$, $H^0((\tilde\fa^{rs})_a,\bC)$
carry an action of the braid group
$B_N\is\pi_1^{W}(\ars,a)$. Let \[
H^i(X_{m,a},\bC)=\bigoplus_{\chi\in I_N^\vee}H^i(X_{m,a},\bC)_\chi\ \ \quad\quad  
H^i(\mY_{m,a},\bC)=\bigoplus_{\chi\in I_N^\vee} H^i(\mY_{m,a},\bC)_\chi\]
 \[  
H^0((\tilde\fa^{rs})_a,\bC)=\bigoplus_{\chi\in I_N^\vee} V_\chi
\] 
be the 
decompositions with respect to the action of $I_N$; for the last identity we recall that $\tilde\fa^{rs}\to \fa^{rs}$ is an $I_N$-torsor.
For $\chi\in I_N^\vee$ and $b\in B_N$, we write $b\cdot\chi$ for the action of $b$ on $\chi$. 
Then the braid group action on $H^i(\mY_{m,a},\bC)$ is described as follows
$$b\in B_N:\,H^i(\mY_{m,a},\bC)_\chi\mapsto H^i(\mY_{m,a},\bC)_{b\cdot\chi}.$$
The $B_N$-actions on $H^i(X_{m,a},\bC)$ and $H^0((\tilde\fa^{rs})_a,\bC)$ are described in the same manner.

By the K\"unneth formula, 
the cohomology of 
the fiber of $\mY_m^t=(\mY_m\times_\ars\tilde\fa^{rs})/I_N$ over $a\in\ars$ is canonically isomorphic to 
\[H^i(\mY^t_{m,a},\bC)\is\bigoplus_{\chi\in I_N^\vee}H^i(\mY_{m,a},\bC)_\chi\otimes V_\chi.\]
Thus by (\ref{X_m=Y_m twist}) we obtain the desired $\pi_1^{K}(\grs,a)\is I_N\rtimes B_N$-equivariant isomorphism.
\end{proof}

The isomorphism in Proposition \ref{mon thm X_m} implies 
the following isomorphism of monodromy representations  
\beq\label{X_m is Y_m}
P(X_m)\is \bigoplus_{\chi\in I_N^\vee} P(\mY_m)_\chi\otimes V_\chi,
\eeq
where $$P(\mY_m):=H^{N-m-1}_{prim}(\mY_{m,a},\bC),\ \ P(\mY_m)_\chi:=H^{N-m-1}(\mY_m,\bC)_\chi\cap P(\mY_m).$$
Our goal is to decompose the representation above 
into irreducible representations.  
Observe that each summand $P(\mY_m)_\chi$ is invariant under the 
action of the pure braid group $P_N$. According to \cite[Theorem 2.5.1]{T}, 
there is an isomorphism of representations of $P_N$
\beqn
P(\mY_m)_\chi\is
\wedge^{N-m-1}H^1(C_a,\bC)_\chi\is\wedge^{N-m-1}H^1(C_{a,\chi},\bC),
\eeqn 
where $P_N$ acts on $H^1(C_{a,\chi},\bC)$ via the map $\rho_{C_\chi}:P_N\ra Sp(2i-2)$, $i=|\chi|/2$.

As an $Sp(2i-2)$-representation $\wedge^{N-m-1}H^1(C_{a,\chi},\bC)$ decomposes into a direct sum of fundamental representations in a well-known manner. This implies the following decomposition of $P(\mY_m)_\chi$ into irreducible representations of $P_N$:
\beq\label{decomp of Y_m}
P(\mY_m)_\chi\is
\wedge^{N-m-1}H^1(C_{a,\chi},\bC)=\bigoplus_{j\equiv N-m-1\on{mod}2,\ j\in[0,l]}\on{P}^j_\chi\eeq
where $l=\min\{N-m-1,-N+m+|\chi|-1\}$.

Combining (\ref{X_m is Y_m}) with (\ref{decomp of Y_m}), we obtain the following decomposition 
\beq\label{X_m= C_chi}
P(X_m)\is\bigoplus_{\chi\in I_N^\vee} P(\mY_m)_\chi\otimes V_\chi\is
\bigoplus_{\chi\in I_N^\vee}\bigoplus_j\on{P}^j_\chi
\otimes V_\chi.
\eeq

Using the notation from the beginning of this section the decomposition \eqref{X_m= C_chi} can be rewritten as 
\beq\label{Primitive coh of X_m}
P(X_m)\is
\bigoplus_{i}^{}\bigoplus_{j\equiv N-m-1\on{mod}2,\ j\in[0,l]} E_{ij}^N\ ,
\eeq
where $N-m+1\leq 2i\leq N$ and 
$l=\min\{N-m-1,-N+m+2i-1\}$.

We have the following 
\begin{lemma}\label{E_ij}
\hspace{2em}
\begin{enumerate}
\item Each $E_{ij}^{N}$ is an irreducible 
representation of $\pi^{K}_1(\fg_1^{rs},a)$. We denote by $\rho^N_{ij}:
\pi^{K}_1(\fg_1^{rs},a)\ra GL(E_{ij}^N)$ the corresponding map.\\ 
\item 
Suppose $j>0$.
Let $H:=\overline{\rho^N_{ij}(P_N)}\subset GL(E_{ij}^N)$ be the
Zariski closure of $\rho^N_{ij}(P_N)$ in $GL(E_{ij}^N)$ (recall $P_N\subset\pi^{K}_1(\fg_1^{rs},a)$ is the pure braid group).
Then we have $\Lie\, H\is\mathfrak{sp}(2i-2)$. 
In particular, the image $\rho^N_{ij}(\pi^{K}_1(\fg_1^{rs},a))$ is infinite.
\end{enumerate}
\end{lemma}
\begin{proof}
We begin with the proof of (1). 
We first show that $E_{ij}^N$ is a $\pi^{K}_1(\fg_1^{rs},a)$-invariant subspace of $P(X_m)$.
For this, we observe that 
the decomposition in (\ref{X_m= C_chi}) is compatible with the action of 
$B_N$, that is, $\text{ for }b\in B_N$, 
$$b:\on{P}^j_\chi
\otimes V_\chi\mapsto\on{P}^j_{b\cdot \chi}
\otimes V_{b\cdot\chi}.$$ 
Since
the braid 
group $B_N$ acts transitively on the set
$\{\chi\in I_{N}^\vee\,|\,|\chi|=2i\}$, it follows that the subspace 
$$E_{ij}^N=\bigoplus_{\chi\in I_{N}^\vee,|\chi|=2i}\on{P}^j_\chi
\otimes V_\chi$$ is stable under the action of $\pi^{K}_1(\fg_1^{rs},a)$. Now since 
each summand 
$\on{P}^j_\chi
\otimes V_\chi$ is irreducible as a representation of 
$P_N$, it follows that each $E_{ij}^N$ is an
irreducible 
representation of $\pi^{K}_1(\fg_1^{rs},a)$.

We prove (2). 
For each $\chi\in I_N^\vee$, we define $\rho_\chi:P_N\xrightarrow{\rho} I_N\xrightarrow{\chi}\mu_2$. Here $\rho$ is the map in \eqref{the map rho}.
Define $$\psi_1:=(\rho_{C_\chi},\bigoplus_{\chi, |\chi|=2i}\rho_\chi)
:P_N\ra Sp(2i-2)\times\mu_2^{\binom {N}{2i}}.$$
Let $V_{ij}$ denote the irreducible representation of $Sp(2i-2)$ associated to the fundamental weight $\omega_j$.
Then the restriction of $\rho_{ij}^N:\pi^{K}_1(\fg_1^{rs},a)\ra GL(E_{ij}^N)$ to $P_N$ can be identified with 
\[\psi:P_N\xrightarrow{\psi_1} Sp(2i-2)\times\mu_2^{\binom {N}{2i}}\xrightarrow{\psi_2}GL(V_{ij})^{\times{N\choose 2i}}\]
where $\psi_2$ maps $Sp(2i-2)$ diagonally into  
$GL(V_{ij})^{\times{N\choose 2i}}$ and $\psi_2$ maps $\mu_2=\{\pm1\}$ to $\pm id\in GL(V_{ij})$. 
Since $\overline{\rho_{C_\chi}(P_N)}=Sp(2i-2)$, it implies that the 
connected component 
$\overline{\psi_1(P_N)}^0=Sp(2i-2)$. 
So to prove (2), it suffices to show that $\Lie(\on{Im}(\psi_2))\is\mathfrak sp(2i-2)$ for $j>0$. This follows
from the fact that the induced map $d\psi_2:\mathfrak sp(2i-2)\ra\bigoplus\mathfrak{gl}(V_{ij})$ on the Lie algebras 
 is injective.  
 \end{proof}

It follows from the lemma above that 
(\ref{Primitive coh of X_m}) is the decomposition of the monodromy representation 
$P(X_m)$ into irreducible subrepresentations. This completes the proof of Theorem
\ref{irred}.

\subsection{Proof of Theorem \ref{irred tE}}
The proof is similar to the case of $X_m$. First using the 
isomorphism (\ref{tX_m=tY_m}) and the same argument as in the case of $X_m$,
we obtain the following proppsition.

\begin{prop}\label{mon thm tX_m}
There is an isomorphism of $I_{N+1}\rtimes B_N$-representations 
\beqn
H^i(\tX_{m,a},\bC)\is\bigoplus_{\chi\in I_{N+1}^\vee}H^i(\tilde\mY_{m,a},\bC)_{\chi}\otimes V_{\chi},
\eeqn
where for $V_\chi$, we regard $\chi$ as an element in $I_N^\vee$
via the map  
$\check{\kappa}:I_{N+1}^\vee\ra I_N^\vee$ in \eqref{characters}, and the group 
$I_{N+1}$ acts on the summand $H^i(\tilde\mY_{m,a},\bC)_\chi\otimes V_\chi$
via the character $\chi\in I_{N+1}^\vee$.

\end{prop}

Set 
$$P(\tilde\mY_m):=H^{N-m-1}_{prim}(\tilde\mY_{m,a},\bC).$$
By Proposition \ref{mon thm tX_m}, there is an isomorphism 
of $I_{N+1}\rtimes B_N$-representations 
\beq\label{tX_m is tY_m}
P(\tX_m)\is
\bigoplus_{\chi\in I_{N+1}^\vee} P(\tilde\mY_m)_\chi\otimes V_\chi.
\eeq
For any $\chi\in I_{N+1}^\vee$
with 
$|\chi|=2i$, 
let $\tC_{a,\chi}$ be the 
hyperelliptic curve defined 
in \S\ref{curves} and let
$\rho_{\tC_\chi}:P_{N}\ra Sp(H^1(\tC_{a,\chi},\bC))\is Sp(2i-2)$ denote 
the monodromy representation for the family 
$\tC_\chi\ra\ars$.  
Again by \cite{T}, we have 
an isomorphism of $P_N$-representations 
\beq\label{iso of tT}
P(\tilde\mY_m)_\chi\is
\wedge^{N-m-1}H^1(\tC_a,\bC)_\chi\is\wedge^{N-m-1}H^1(\tC_{a,\chi},\bC)).
\eeq 
Combining \eqref{tX_m is tY_m} with \eqref{iso of tT} we obtain the following 
decomposition 
\begin{equation}\label{isom tX_m}
P(\tX_m)\is\bigoplus_{\chi\in I_{N+1}^\vee}\wedge^{N-m-1}H^1(\tC_{a,\chi},\bC)\otimes V_\chi.
\end{equation}
We describe the monodromy representation 
$ P(\tX_m)^{\sigma=-id}$ (recall that $\sigma$ is the involution on $\tilde X_{m}$).
For this, we first 
observe that the involution action of $\langle\sigma\rangle\is\bZ/2\bZ$ on $\tX_m$ is 
equal to the composition of \[i_\infty:\bZ/2\bZ\ra I_{N+1},\ 1\mapsto {(0,...,0,1)},\] with the action of $I_{N+1}$ on $\tX_m$. Hence 
by (\ref{isom tX_m})
we have 
\beq\label{decomp-tX}
P(\tX_m)^{\sigma=-id}=\bigoplus_{\substack{\chi\in I_{N+1}^\vee\\ N+1\in\on{supp}\chi}} P(\tX_m)_\chi\is
\bigoplus_{\substack{\chi\in I_{N+1}^\vee\\ N+1\in\on{supp}\chi}} 
\wedge^{N-m-1}H^1(\tC_{a,\chi},\bC)\otimes V_\chi.
\eeq
 Again, since $\rho_{\tC_\chi}(P_N)$ is Zariski dense in $Sp(H^1(\tC_\chi,\bC))$, we have the following decomposition 
 $$\wedge^{N-m-1}H^1(\tC_{a,\chi},\bC)=\bigoplus_{ j\equiv N-m-1\on{mod}2,\ j\in[0,l]}\tilde{\on{P}}^j_\chi,$$
where $l=\min\{N-m-1,-N+m+|\chi|-1\}$.

Using the notation from the beginning of this section
the decomposition \eqref{decomp-tX} can be rewritten as\[P(\tX_m)^{\sigma=-id}\is
\bigoplus_{i}^{}\bigoplus_{\substack{j\equiv N-m-1\on{mod}2\\ j\in[0,l]}}^{}\tE_{ij}^N,\]
where $N-m+1\leq 2i\leq N+1$, 
$l=\min\{N-m-1,-N+m+2i-1\}$

The same argument as in the proof of Lemma \ref{E_ij} shows the following
\begin{lemma}\label{tE_ij}
\hspace{2em}
\begin{enumerate}
\item
$\tE_{ij}^N$ is an irreducible representation of $\pi_1^K(\grs,a)$. 
We denote by $\tilde\rho^N_{ij}:\pi_1^K(\grs,a)\ra GL(E_{ij}^N)$
the corresponding map. \\
\item Suppose $j>0$.
Let $H:=\overline{\tilde\rho^N_{ij}(P_N)}\subset GL(\tE_{ij}^N)$ be the
Zariski closure of $\tilde\rho^N_{ij}(P_N)$ in $GL(\tE_{ij}^N)$.
Then we have
$\Lie\, H\is\mathfrak{sp}(2i-2)$. 
In particular, the image $\tilde\rho^N_{ij}(\pi^{K}_1(\fg_1^{rs},a))$ is infinite.
\end{enumerate}
\end{lemma}

This completes the proof of Theorem \ref{irred tE}.

\subsection{The local systems $E_{ij}^N$ and $\tE_{ij}^N$}In this subsection, we show that from the constructions in previous sections, we have obtained the following set consisting of pairwise non-isomorphic irreducible $K$-equivariant local systems on $\Lg^{rs}_1$
\beq\label{set of local systems}
\left\{E_{ij}^{2n+1},\,i\in[1,n],\,j\in[0,i-1];\ \tE^{2n+1}_{ij},\,i\in[1,n+1],\,j\in[1,i-1],\,\tE^{2n+1}_{n+1,0}\cong\bC\right\}.\eeq
For this, we first observe that 
$$\text{$E_{ij}^{N}\is E_{i'j'}^{N}$ and $\tE_{ij}^N\is\tE_{i'j'}^N$ if and only if $i=i',j=j'$.}$$
In fact, assume that $E_{ij}^N\is E_{i'j'}^N$. Then we must have $i=i'$, otherwise,
the centralizer $Z_{K}(a)\cong I_N$ would act differently on $E_{ij}^N$ and $E^N_{i'j'}$. Now regarding $E_{ij}^N$ and $E^N_{ij'}$ as $P_N$-representations we see that $j=j'$. Similar argument applies to $\tE_{ij}^N$.

It remains to prove the following.
\begin{lemma}\label{E and tE}
We have $E_{i,j}^{N}\cong\tE_{i',j'}^{N}$ if and only if  $i+i'=(N+1)/2$ and $j=j'=0$.
\end{lemma}
\begin{proof}
Recall 
$$E_{ij}^{N}=\bigoplus_{\chi\in I_{N}^\vee,|\chi|=2i}\on{P}^j_\chi
\otimes V_\chi,\,\, \tE_{i'j'}^{N}=\bigoplus_{\substack{\chi'\in I_{N+1}^\vee,\, |\chi|=2i',\\  N+1\in\on{supp}\chi}}\tilde{\on{P}}^{j'}_{\chi'}\otimes V_{\chi'}.$$
For $V_{\chi'}$ we regard $\chi'$ as an element in $I_N^\vee$ via the map 
$\check{\kappa}:I_{N+1}^\vee\ra I_{N}^\vee$ in \eqref{characters}.
Observe that for $\chi'\in I_{N+1}^\vee$ and $N+1\in\on{supp}\chi'$, we have
\beq\label{chi and chi'}
\check{\kappa}(\chi')=\chi\text{ if and only if }\on{supp}\chi=\{1,\ldots,N+1\}\backslash\on{supp}\chi'.
\eeq
Thus the map $\check{\kappa}$ maps the subset $\{\chi'\in I_{N+1}^\vee,\, |\chi'|=2i',\ N+1\in\on{supp}\chi'\}$ bijectively to the subset $\{\chi\in I_{N}^\vee,|\chi|=2i\}$ where $2i+2i'=N+1$. 
Hence we have 
\beq\label{iso for chi}
\bigoplus_{\chi\in I_{N}^\vee,|\chi|=2i}V_{\chi}\cong\bigoplus_{\substack{\chi\in I_{N+1}^\vee,\, |\chi'|=2i',\\  N+1\in\on{supp}\chi}}V_{\chi'}\,\,\,\text{if and only if $i+i'=(N+1)/2$}.
\eeq
This implies $E_{i,0}^{N}\cong\tE_{i',0}^{N}$ for $i+i'=(N+1)/2$.

Conversely, we observe that 
$E_{i,j}^{N}\cong\tE_{i',j'}^{N}$ implies  
$\on{P}_\chi^j\otimes V_\chi\cong\tilde{\on{P}}_{\chi'}^{j'}\otimes V_{\chi'}$ (as representations of $I_N\rtimes P_N$)
for some $\chi\in I_{N}^\vee$ and  $\chi'\in I_{N+1}^\vee$ with $N+1\in\on{supp}\chi'$.
This implies that $\check\kappa(\chi')=\chi$ and 
it follows from \eqref{chi and chi'} that
$\on{supp}\chi\cap\on{supp}\chi'=\phi$. 
Therefore the monodromy representation of the restriction of 
$C_\chi\ra\ars$ (resp. $\tC_{\chi'}\ra\ars$) to the subvariety $\ars(a,\chi')$ (resp. $\ars(a,\chi)$) in \eqref{subvariety of ars} 
is trivial. On the other hand,  
the monodromy representation of the restriction of $\tC_{\chi'}\ra\ars$ (resp. $C_\chi\ra\ars$) to 
$\ars(a,\chi')$ (resp. $\ars(a,\chi)$) has Zariski dense image (see \eqref{The monodromy representation is Zariski dense}). 
This forces $j=j'=0$ and the desired claim follows again from \eqref{iso for chi}.
\end{proof}

\subsection{The local systems $E^{2n+1}_{i0}$, $\tE_{n+1,j}^{2n+1}$ and the $\calL_i$'s, $\cF_i$'s in \cite{CVX} }
Recall that in \cite[\S 2.3]{CVX}, we have defined the local systems $\calL_i$ and $\cF_i$ on $\Lg_1^{rs}$. We have the following.
\begin{lemma}
We have
\beq\label{E and L}
E_{i,0}^{2n+1}\cong \calL_{2i}\ \ \emph{if}\ \ 1\leq 2i\leq n,\ \ 
E^{2n+1}_{i,0}\cong \calL_{2n-2i+1}\ \  \emph{if}\ \  n+1\leq 2i\leq 2n,
\eeq
\beq\label{tE and F}
\tE_{n+1,j}^{2n+1}\cong\cF_j\text{ for } 1\leq j\leq n.\ \ 
\eeq
\end{lemma}

\begin{proof}

We begin with the proof of \eqref{E and L}. Recall from {\em loc. cit.} that we have
\beq\label{L_i's}
(\check{\pi}_{2^n1})_*\bC|_{\Lg_1^{rs}}\cong \bigoplus_{i=0}^n\calL_i\text{ and }\dim\calL_i={2n+1\choose i}
\eeq
where 
\beqn
\check{\pi}_{2^n1}:K\times^{P_K}[\fn_P,\fn_P]_1^\bot:=\{(x,0\subset V_n\subset V_n^\p\subset\bC^{2n+1})\,|\,x\in\Lg_1,\,xV_n\subset V_n^\p\}\to\Lg_1.
\eeqn
On the other hand, recall the $I_N$-torsor over $\fa^{rs}$ in \S\ref{branched cover X_m}:
\[\tilde{\pi}:\tilde\fa^{rs}=\{(a,c)\,|\,a=(a_1,...,a_{N})\in\fa^{rs},\,c=(c_1,...,c_N),\, c^2_i=
\prod_{j\neq i}(a_j-a_i)\}/(\bZ/2\bZ)\to\fa^{rs}. \] 
We have
\beq\label{E_i0's}
\tilde\pi_*\bC|_{\Lg_1^{rs}}\cong\bC\oplus\bigoplus_{i=1}^nE^{2n+1}_{i,0}\text{ and }\dim E^{2n+1}_{i,0}={2n+1\choose 2i}.
\eeq
We show that there is a $I_N\rtimes W$-equivariant isomorphism
\beq\label{I_N torsors}
\tilde\fa^{rs}\cong K\times^{P_K}[\fn_P,\fn_P]_1^\bot|_{\fa^{rs}}.
\eeq
Then \eqref{E and L} follows from \eqref{L_i's}, \eqref{E_i0's}, and dimension considerations of
the representations.

Using the identities  
$\sum_{i=1}^{2n+1} a_i^k\,c_i^{-2}=0,\ k=0,...,2n-1,$
it is easy to check that the map 
\[\tilde\fa^{rs}\ra X_{2n}|_{\fa^{rs}},\ (a,c)\in\tilde\fa^{rs}\mapsto [c_1^{-1},...,c_{2n+1}^{-1}]\]
defines a $I_N\rtimes W$-equivaraint isomorphism 
\[\tilde\fa^{rs}\cong X_{2n}|_{\fa^{rs}}.\]
On the other hand, by the description of the
Hessenberg varieties $\on{Hess}_{n}^{E,\bot}$ in \S\ref{two Hessenberg}, 
we have a natural map 
\[\on{Hess}_{n}^{E,\bot}\to K\times^{P_K}[\fn_P,\fn_P]_1^\bot,\ (x, V_{n-1}\subset V_n)\mapsto(x,V_n). \]
It is easy to check that the map above is a $K$-equivariant isomorphism over $\Lg_1^{rs}$.
The desired isomorphism \eqref{I_N torsors} follows from the following compositions of  isomorphisms
\[\tilde\fa^{rs}\cong X_{2n}|_{\fa^{rs}}\stackrel{\on{Thm} \ref{H=X}}\cong\on{Hess}_{n}^{E,\bot}|_{\fa^{rs}}\cong K\times^{P_K}[\fn_P,\fn_P]^\bot|_{\fa^{rs}}.\]
This completes the proof of \eqref{E and L}.

To prove \eqref{tE and F}, we observe that
\beqn
\tE^{2n+1}_{n+1,j}\cong\tilde{P}^j_{\chi_0}\otimes V_{\chi_0}\cong(\wedge^j H^1(\tC_{a,\chi_0},\bC))_{prim}\otimes V_{\chi_0},
\eeqn
where $\chi_0\in I_{N+1}^\vee$ is the unique character such that 
$|\chi_0|=2n+2$, and $\tC_{a,\chi_0}$ is 
the  
hyperelliptic curve of genus $n$ with affine equation $y^2=\prod_{i=1}^{2n+1}(x-a_i)$. 
By \eqref{chi and chi'},  
$\chi_0$, when regarded as an element in $I_N^\vee$ (see \eqref{characters}), is trivial. Hence 
$I_N$ acts trivially on $\tE_{n+1,j}^{2n+1}$ and $V_{\chi_0}$, i.e.
\beq\label{E_n+1,j's}
\tE^{2n+1}_{n+1,j}\cong(\wedge^j H^1(\tC_{a,\chi_0},\bC))_{prim}\cong\cF_j
\eeq
where the last isomorphism follows from the discussion above and the definition of $\cF_j$'s in {\em loc.cit.}

\end{proof}

\begin{remark}In \cite[Proof of Proposition 4.3]{CVX} we used the fact that among the \linebreak$\IC(\Lg_1,\calL_i)$'s $(i\geq 1)$, only $\IC(\Lg_1,\calL_{2j-1})$, $1\leq j\leq m$, appear in the decomposition of $(\check\tau_m)_*\bC[-]$, where $\check\tau_m=\check\tau_m^{2n+1}$ and $2m\leq n+1$. To prove this fact, it suffices to show that in the decomposition of the monodromy representation $P(X_{2m})$, only the above mentioned local systems appear. Applying Theorem \ref{irred} to $P(X_{2m}) $ with $N=2n+1$ we see that among the $E_{i0}$'s only those with $n-m+1\leq i\leq n$ appear. The desired conclusion follows from \eqref{E and L} and the fact that $2m\leq n+1$.
\end{remark}

\section{Computation of the Fourier transforms}\label{computation of FT}
Let $\fF:D_K(\Lg_1)\to D_K(\Lg_1)$ denote the Fourier transform, where we identify $\Lg_1$ and $\Lg_1^*$ via a $K$-invariant non-degenerate  bilinear form on $\Lg_1$. The Fourier transform $\fF$ induces an equivalence of categories $\fF:\on{Perv}_K(\Lg_1)\to\on{Perv}_K(\Lg_1).$

In this section we 
study the 
Fourier transforms of 
$\IC(\fg_1,E_{ij}^N)$ and $\IC(\fg_1,\tE_{ij}^N)$.
We show that they are supported on
$\mathcal N_1$, more precisely, on $\mN_1^3\subset\mN_1$,  
the closed subvariety consisting of nilpotent elements of 
order less than or equal to $3$. Thus we obtain many  more examples of IC complexes supported on nilpotent orbits whose Fourier transforms have both full support and 
infinite monodromy (see also \cite{CVX}). As an interesting corollary (see Example \ref{hyperelliptic curve of odd}), we show that 
the 
Fourier transform of the $\IC$ extension of the unique  non-trivial irreducible $K$-equivariant local system on
the minimal nilpotent orbit has full support and its monodromy is given by
a 
universal family of hyperelliptic curves.

The main result of this section is the following theorem.
\begin{thm}\label{FT of E_ij}
Let $\mN_1^3\subset\mN_1$ be 
the closed subvariety consisting of nilpotent elements of 
order less than or equal to $3$. Then 
$\mathfrak F(\IC(\fg_1,E_{ij}^{N}))$  and $\mathfrak F(\IC(\fg_1,\tE_{ij}^N))$
are supported on $\mN^3_1$.
\end{thm}

We first argue the case $\mathfrak F(\IC(\fg_1,E_{ij}^{N}))$.
For $m\leq \frac{N-1}{2}$,
consider the families of Hessenberg varieties $$\sigma_m^N:\on{Hess}_m^O\ra\fg_1,\ \tau_m^N:\on{Hess}_m^E\ra\fg_1$$ and 
$$\check\sigma_m^N:\on{Hess}_m^{O,\p}\ra\fg_1,\ \check\tau_m^{N}:\on{Hess}_m^{E,\p}\ra\fg_1$$ 
defined in \S\ref{two Hessenberg}. We have (see \eqref{Fourier})
\beqn
\fF((\check\sigma_m^N)_*\bC[-])=(\sigma_m^N)_*\bC[-],\ \ \fF((\check\tau_m^N)_*\bC[-])=(\tau_m^N)_*\bC[-].
\eeqn
By Theorem \ref{H=X}, over $\fg_1^{rs}$,
we have $\on{Hess}_m^{O,\p}\is X_{2m-1}$, $\on{Hess}_m^E\is X_{2m}$. Hence
the decomposition theorem implies that
\beqn
 \IC(\fg_1,P(X_{2m-1}))
\text{ is a direct summand of } (\check\sigma_m^N)_*\bC[-]\eeqn
\beqn
 \IC(\fg_1,P(X_{2m})))
\text{ is a direct summand of } (\check\tau_m^N)_*\bC[-]\,.
\eeqn
Therefore the Fourier transforms $\mathfrak F(\IC(\fg_1,P(X_{2m-1})))$ 
and $\mathfrak F(\IC(\fg_1,P(X_{2m})))$
appear as direct summands of $(\sigma_m^N)_*\bC[-]$ and $(\tau_m^N)_*\bC[-]$. 
Now in view of (\ref{support tau}) and (\ref{support sigma}), we see that 
$\mathfrak F(\IC(\fg_1,P(X_{2m-1})))$ 
and $\mathfrak F(\IC(\fg_1,P(X_{2m})))$ are supported on 
$\mN_1^3$. Since each local system $E^N_{ij}$ appears in $
P(X_m)$ for some $m$ (see  Theorem \ref{irred}), we conclude that $\mathfrak F(\IC(\fg_1,E_{ij}^{N}))$ is supported on $\mN_1^3$.

It remains to consider the case $\mathfrak F(\IC(\fg_1,\tE_{ij}^N))$. 
Since each local system $\tE_{ij}^N$ appears in $P(\tX_m)^{\sigma=-id}$
for some $m$, we are reduced to proving the following proposition:
\begin{proposition}\label{prop FT of tX}
$\mathfrak F(\IC(\fg_1,P(\tX_{m})^{\sigma=-id}))$ is supported on
$\mN_1^3$.
\end{proposition}

The proof of this proposition occupies the remainder of this section.

\subsection{Proof of Proposition \ref{prop FT of tX} when $m$ is odd}\label{sec case of Podd}

Recall that in \S\ref{two Hessenberg} we have introduced the families of Hessenberg varieties
\beqn
\Hess_k^{O,\p}:=\{(x,0\subset V_{k-1}\subset V_k\subset V_k^\bot\subset V_{k-1}^\bot\subset V=\bC^N)\,|\,x\in\fg_1,\ xV_{k-1}\subset V_k\}
\eeqn
\beqn
{\Hess_k^{E,\p}}=\{(x,0\subset V_{k-1}\subset V_k\subset V_k^\bot\subset V_{k-1}^\bot\subset V=\bC^N)\,|\,x\in\fg_1,\ xV_{k-1}\subset V_k, xV_k\subset V_k^\bot\}
\eeqn
and the natural projection maps
$\check{\sigma}_k^N:\Hess_k^{O,\p}\to\Lg_1$, 
$\check{\tau}_k^N:{\Hess_k^{E,\p}}\to\Lg_1$.

Our first goal is to show that  $\IC(\fg_1,P(\tX_{2k-1}))$ appears as a direct summand in the 
push forward of certain intersection cohomology complex on $\Hess_k^{O,\p}$ along $\check{\sigma}_k^N$.  

Let $[\bG_a/\bG_m^{[2]}]$
be the stack quotient, where $\bG_m^{[2]}\cong\bG_m$ acts on $\bG_a$ via the 
square map, i.e., for $t\in\bG_m$ and $x\in\bG_a$, $t:\,x\mapsto t^2x$.
We first introduce a map $$\alpha:\Hess_k^{O,\p}\ra [\bG_a/\bG_m^{[2]}].$$ Recall that
such a map is equivalent to a pair $(\widetilde\Hess_k^{O,\p},\phi)$, where 
$\widetilde\Hess_k^{O,\p}$ is a $\bG_m$-torsor over $\Hess_n^{O,\p}$ and $\phi:
\widetilde\Hess_k^{O,\p}\ra\bG_a$ is a 
map such that 
\beq\label{cond}
\text{$\phi(t\cdot v)=t^2\phi(v)$ for $v\in\widetilde\Hess_k^{O,\p}$ and 
$t\in\bG_m$.}
\eeq
 To construct such a pair, we set
\[\widetilde\Hess_k^{O,\p}:=\{(x,V_{k-1}\subset V_k,l)\,|\,(x,V_{k-1}\subset V_k)\in \Hess_k^{O,\p},\ 0\neq l\in V_k/V_{k-1}\is\bC\},\]
where the action of $\bG_m$  on $\widetilde\Hess_k^{O,\p}$ is given by  
$t\cdot(x,V_{k-1}\subset V_k,l)=(x,V_{k-1}\subset V_k,tl)$ for $t\in \bG_m$. Define  
\[\phi:\widetilde\Hess_k^{O,\p}\ra\bG_a,\ (x,V_{k-1}\subset V_k,l)\mapsto\langle xl,l\rangle_Q.\]
Note that the above pairing is well-defined since $xV_{k-1}\subset V_k$ and $xV_{k}\subset V_{k-1}^\bot$. One  checks easily that $\phi$ satisfies \eqref{cond}. This finishes the construction of 
$(\widetilde\Hess_k^{O,\p},\phi)$, hence that of the map $\alpha:\Hess_k^{O,\p}\ra[\bG_a/\bG_m^{[2]}].$ 
By construction, the map $\alpha$ 
is $K$-equivaraint (where $K$ acts trivially on $[\bG_a/\bG_m^{[2]}]$), moreover  it
factors through $\Hess_k^{E,\p}$, i.e.,
\beq\label{bar a}
\alpha:\Hess_k^{O,\p}\xrightarrow{}\Hess_k^{O,\p}/\Hess_k^{E,\p}\xrightarrow{\bar \alpha}
[\bG_a/\bG_m^{[2]}].\
\eeq
There is a unique non-trivial irreducible local system $\mL$ on 
$[\bG_m/\bG_m^{[2]}]\subset[\bG_a/\bG_m^{[2]}]$. We denote by $\IC([\bG_a/\bG_m^{[2]}],\mL)$ the corresponding 
intersection cohomology complex on $[\bG_a/\bG_m^{[2]}]$.
Let \[\mathcal K:=(\check\sigma_k^N)_*\alpha^*\IC([\bG_m/\bG_m^{[2]}],\mL)\in D_K(\fg_1).\]
The factorization in (\ref{bar a}) and 
the functorial properties of Fourier transform (see  \cite[Proposition 3.7.14]{KaS})  
imply the following
\beq\label{supp of K}
\text{$\mathfrak F(\mathcal K)$ is supported on 
$\on{Im}\,(\tau_k^N)\subset\cN_1^3$.}
\eeq
Thus to show that $\fF(\IC(\fg_1,P(\tX_{2k-1})))$ 
is supported on $\mathcal N_1^3$, 
it suffices to show that
\beq\label{odd}
\text{the complex $\mathcal K$ contains $\IC(\fg_1,P(\tX_{2k-1})^{\sigma=-id})$ as a direct summand.}
\eeq
Let $$\tilde\pi_{2k-1}:\tX_{2k-1}\xrightarrow{p_{2k-1}}X_{2k-1}\xrightarrow{\pi_{2k-1}}\grs$$ be the branched double cover of $X_{2k-1}$ and $\sigma$ the involution on $\tX_{2k-1}$ defined in \S\ref{quadrics}.
We have that  
$((\tilde\pi_{2k-1})_*\bC)^{\sigma=-id}$ contains $P(\tX_{2k-1})^{\sigma=-id}$
as a direct summand. 
The statement \eqref{odd} follows from the following claim
\[\mathcal K|_{\grs}\is((\tilde\pi_{2k-1})_*\bC)^{\sigma=-id}.
\]
To prove the claim,
let $s:[\bG_a/\bG_m]\ra [\bG_a/\bG_m^{[2]}]$
be the descent of the map $\bG_a\ra\bG_a,\,t\mapsto t^2$. Then 
from the definitions of $\tX_{2k-1}$ and the map $\alpha$, one can check that,
under the isomorphism $X_{2k-1}\is \on{Hess}_k^{O,\perp}|_{\Lg_1^{rs}}$ in Theorem \ref{H=X},
the branched double cover $\tX_{2k-1}$ can be identified with the following 
fiber product 
\beqn
\xymatrix{\tX_{2k-1}\ar[d]^{p_{2k-1}}\ar[r]&[\bG_a/\bG_m]\ar[d]^s
\\X_{2k-1}\is \on{Hess}_k^{O,\perp}|_{\grs}\ar[r]^-{\alpha|_\grs}&[\bG_a/\bG_m^{[2]}].}
\eeqn
Since $$s_*\bC=(s_*\bC)^{\sigma=id}\oplus(s_*\bC)^{\sigma=-id}=
\bC\oplus\IC([\bG_a/\bG_m^{[2]}],\mL),$$ by proper base change we have 
\[(\alpha|_\grs)^*\IC([\bG_a/\bG_m^{[2]}],\mL)\is ((p_{2k-1})_*\bC)^{\sigma=-id}.\] This implies that \[\mathcal K|_{\grs}\is(\pi_{2k-1})_*(\alpha|_\grs)^*\IC([\bG_a/\bG_m^{[2]}],\mL)\is
((\tilde\pi_{2k-1})_*\bC)^{\sigma=-id}.\]
This proves \eqref{odd}.

\begin{remark}
The construction of the map $\alpha$ was inspired by discussions with Zhiwei Yun. In particular, the idea of making use of the local system $\calL$ on $[\bG_a/\bG_m^{[2]}]$ was explained to one of us by him.
\end{remark}

\subsection{Proof of Proposition \ref{prop FT of tX} when $m$ is even}

Let us consider the following family of  
Hessenberg varieties 
\begin{eqnarray*}
&&H=\{(x,0\subset V_{k-1}\subset V_k\subset V_{k+1}\subset V_{k+1}^\bot
\subset V_{k}^\bot
\subset V_{k-1}^\bot\subset V=\bC^N)\\
&&\qquad\qquad|\,x\in\fg_1,\ xV_{k-1}\subset V_k,
xV_{k}\subset V_k^\bot\}.
\end{eqnarray*}
Note that  
the natural map \[p:H\ra \Hess_k^{E,\p},\  (x, V_{k-1}\subset V_k\subset V_{k+1})\mapsto
(x, V_{k-1}\subset V_k)\] realizes $H$ as a quadric bundle over $\Hess_k^{E,\p}$.

We first construct a map 
$\beta:H\ra[\bG_a/\bG_m^{[2]}]$. The construction is very similar to 
that of the map $\alpha$ in \S\ref{sec case of Podd} and we use the notations there.
Set 
\[\widetilde H:=\{(x,V_{k-1}\subset V_k\subset V_{k+1},l)\,|\,(x,V_{k-1}\subset V_k\subset V_{k+1})\in H,\ 0\neq l\in V_k/V_{k-1}\is\bC\},\]
where the action of $\bG_m$ on $\widetilde H$ is given by  
$t\cdot(x,V_{k-1}\subset V_k\subset V_{k+1},l)=(x,V_{k-1}\subset V_k
\subset V_{k+1},tl)$. Define 
\[\phi:\widetilde H\ra\bG_a,\ (x,V_{k-1}\subset V_k\subset V_{k+1},l)\mapsto\langle xl,xl\rangle_Q.\]
Note that the above pairing is well-defined since $xV_{k-1}\subset V_k$ and $xV_{k}\subset V_{k}^\bot$. One checks that $\phi$ satisfies \eqref{cond}. This finishes the construction of 
$(\widetilde H,\phi)$. Hence we obtain a map $\beta:H\ra[\bG_a/\bG_m^{[2]}]$.

Let $f:H\ra\fg_1$ be the natural projection map. 
Define 
\[\mathcal F:=f_*\beta^*\IC([\bG_a/\bG_m^{[2]}],\mL)\in D_K(\fg_1).\]
We show that
\beq\label{support F}
\text{$\mathfrak F(\mathcal F)$ is supported on $\mathcal N_1^3$, and}
\eeq
\beq\label{direct summand}
\text{the complex $\mathcal F$ contains $\IC(\fg_1,P(\tX_{2k})^{\sigma=-id})$ as 
a direct summand.}
\eeq
The proposition then follows from \eqref{support F} and \eqref{direct summand}. 

To prove \eqref{support F}, let 
\begin{eqnarray*}
&&H'=\{(x,0\subset V_{k-1}\subset V_k\subset V_{k+1}\subset V_{k+1}^\bot
\subset V_{k}^\bot
\subset V_{k-1}^\bot\subset V=\bC^N)\\
&&\qquad\qquad|\,x\in\fg_1,\ xV_{k-1}\subset V_k,
xV_{k}\subset V_{k+1}\}.
\end{eqnarray*}
Note that $H'\subset H$ is a sub-bundle.
By construction, the map $\beta$ factors through $H'$, i.e., \[\beta:H\stackrel{}\ra H/H'\stackrel{\beta'}\ra
[\bG_a/\bG_m^{[2]}].\]Let $\check f'$ be the natural projection map 
$$\check f':(H')^\bot:=\{(x,0\subset V_{k-1}\subset V_k\subset V_{k+1}\subset V_{k+1}^\bot
\subset V_{k}^\bot
\subset V_{k-1}^\bot\subset\bC^N)$$
$$\quad\quad\quad\,|\,x\in\fg_1,\ xV_{k}=0,\,xV_{k+1} \subset V_{k-1},\,
xV_{k}^\p\subset V_k\}\ra\mathcal N_1.$$ 
A direct calculation shows that 
\beq
\on{Im}\check f'=\bar\cO_{3^k1^{N-3k}}\text{ if }3k\leq N\text{ and }\on{Im}\check f'=\bar\cO_{3^{N-2k}2^{3k-N}}\text{ if }3k\geq N+1.
\eeq
The standard properties of Fourier transform imply that 
\beqn
\text{$\mathfrak F(\mathcal F)$ is supported on $\on{Im}(\check f')\subset\mN_1^3$.}
\eeqn This proves \eqref{support F}.

It remains to prove \eqref{direct summand}. 
Notice that the map $\beta$ factors as 
$\beta:H\stackrel{p}\ra \Hess_k^{E,\p}\xrightarrow{\bar \beta} [\bG_a/\bG_m^{[2]}]$. Consider the following diagram
\[\xymatrix{\beta:H\ar[rd]_f\ar[r]^p&\Hess_k^{E,\p}\ar[d]^{\check\tau_k^N}\ar[r]^-{\bar \beta}&[\bG_a/\bG_m^{[2]}]\\
&\fg_1&}.\]
We have $$\mathcal F:=f_*\beta^*\IC([\bG_a/\bG_m^{[2]}])\is(\check\tau_k^N)_*p_*p^*\bar \beta^*(\IC([\bG_a/\bG_m^{[2]}],\mL))$$
which is isomorphic to 
$(\check\tau_k^N)_*(\bar \beta^*(\IC([\bG_a/\bG_m^{[2]}],\mL))\otimes p_*\bC).$
Since $\bC$ is a direct summand of $p_*\bC$, it implies that
$(\check\tau_k^N)_*(\bar \beta^*(\IC([\bG_a/\bG_m^{[2]}],\mL))$
is a direct summand of $\mathcal F$. 
So it is enough to show that
$$\text{$\IC(\fg_1,P(\tX_{2k})^{\sigma=-id})$
is a direct summand of 
$(\check\tau_k^N)_*(\bar \beta^*(\IC([\bG_a/\bG_m^{[2]}],\mL))$.}$$
This follows from the same argument as in the proof of \eqref{odd}, replacing 
$X_{2k-1}$ (resp. $\tX_{2k-1}$) there by $X_{2k}$ (resp. $\tX_{2k}$).
Thus the proof of the proposition is complete.

\subsection{Matching for $\IC(\bar\cO_{2^i1^{2n+1-2i}},\cE_i)$, $i$ odd} Here we complete the proof of \cite[Theorem 2.3]{CVX} by treating the case of odd $i$. In \cite{CVX} we treated the even case of the proposition below and showed that there exists a permutation $s$ of the set $\{2j+1\,|\,1\leq 2j+1\leq n\}$, such that $\fF(\IC(\bar\cO_{2^i1^{2n+1-2i}},\cE_i))=\IC(\Lg_1,\cF_{s(i)})$ (see Proposition 3.2 and Theorem 2.3 in {\em loc.cit.}).
\begin{proposition}\label{matching nontrivial}
We have that 
\beqn
\fF(\IC(\bar\cO_{2^i1^{2n+1-2i}},\cE_i))=\IC(\Lg_1,\cF_{i}),
\eeqn
where $\cE_i$ denotes the unique non-trivial irreducible $K$-equivariant local system on $\cO_{2^i1^{2n+1-2i}}$.
\end{proposition}
\begin{proof}
It remains to prove the proposition for odd $i$. 
Assume that $2m\leq n+1$. By \eqref{supp of K} and \eqref{odd}, we see that the Fourier transform of $\IC(\Lg_1,P(\tilde{X}_{2m-1})^{\sigma=-id})$ is supported on $\on{Im}\tau_m^N=\bar\cO_{3^{m-1}2^11^{2n+2-3m}}$ (see \eqref{support tau}). Using Theorem \ref{irred tE} and \eqref{tE and F} we obtain that
\beqn
\begin{gathered}
\IC(\Lg_1,\cF_i)\text{ is a direct summand of }\IC(\Lg_1,P(\tilde{X}_{2m-1})^{\sigma=-id})\text{ if and only if}\\
i\text{ is odd and }1\leq i\leq 2m-1.
\end{gathered}
\eeqn
This implies that the Fourier transform of $\IC(\Lg_1,\cF_{2j-1})$, $1\leq j\leq m$, is supported on $\bar\cO_{3^{m-1}2^11^{2n+2-3m}}$. Now it is easy to check that $\cO_{2^i1^{2n+1-2i}}\subset\bar\cO_{3^{m-1}2^11^{2n+2-3m}}$ if and only if $i\leq 2m-1$. In view of \cite[Proposition 3.2 and Theorem 2.3]{CVX}, the proposition follows by induction on $m$.

\end{proof}

\begin{example}\label{hyperelliptic curve of odd}
Let $\mO_{min}=\mO_{2^11^{2n-1}}$. By the above proposition, we have 
\beqn
\mathfrak F(\IC(\fg_1,\mF_1))\is \IC(\bar\mO_{min},\mE_1),
\eeqn
where
$$\mF_1\is\tE_{n+1,1}^{2n+1}\is H^1(\tC_{a,\chi_0},\bC)\text{ (see \eqref{E_n+1,j's})}$$
is isomorphic to the monodromy representation associated with $\bar C_{\chi_0}\ra\fc^{rs}$, the 
universal family of hyperelliptic curves  in \S\ref{curves}.

\end{example}

\section{Conjectures and examples}\label{conjs and examples}
Let $N=2n+1$ and let $E^{2n+1}_{ij}$ (resp. $\tE^{2n+1}_{ij}$) be the monodromy representations of $\pi_1^K(\Lg_1^{rs})$ constructed from the 
families of complete intersections of quadrics in $\mathbb P^{2n}$ (resp. their double covers), see \S\ref{sec-monodromy of Hess}. Let $\{(\mO,\mE)\}_{\leq 3}$ denote the set of  pairs $(\mO,\cE)$ where $\mO$ is a 
$K$-orbit in $\mathcal N_1^3$ and $\cE$ is an irreducible $K$-equivariant local system on 
$\mO$ (up to isomorphism). Using Theorem \ref{FT of E_ij},  we establish  
an injective map
\beq\label{injectionS}
\mathcal{S}:\left\{\begin{array}{c}E_{ij}^{2n+1},\,i\in[1,n],\,j\in[0,i-1];\\
\tE^{2n+1}_{ij},\,i\in[1,n+1],\,j\in[1,i-1],\,\tE^{2n+1}_{n+1,0}\cong\bC\end{array}\right\}\hookrightarrow\{(\mO,\cE)\}_{\leq 3},
\eeq
where $\calS(E_{ij}^{2n+1})=(\cO,\cE)$ if and only if $\fF(\Lg_1,E_{ij}^{2n+1})=\IC(\bar\cO,\cE)$, similarly for $\tE_{ij}^{2n+1}$. Here the $K$-equivariant local systems on $\Lg_1^{rs}$ in the left hand side of \eqref{injectionS} are pairwise non-isomorphic, see \eqref{set of local systems}.

In this section we state two conjectures (Conjecture \ref{conj 1} and Conjecture \ref{conj 2})
that describe  the map $\calS$ in \eqref{injectionS} in the case of $\{E_{ij}^{2n+1}\}$ explicitly.
We verify our conjectures in several examples by studying 
various families of Hessenberg varieties.

In what follows we make use of the following observation:
\beq\label{orbit dim}
\text{an orbit $\cO_{3^k2^l1^{2n+1-3k-2l}}\subset\cN_1^3$ is odd dimensional $\Leftrightarrow k$ is odd and $l$ is even.} 
\eeq
This follows from the fact that $\dim\cO_{3^k2^l1^{2n+1-3k-2l}}=2(k+2kn+ln)-l(l-1)-3k(k+l)$, which one readily deduces from the formula $\dim Z_K(x)=\sum (i-1)\lambda_i$ (\cite{S}), for $x$ in a nilpotent orbit corresponding to the partition $\lambda_1\geq\lambda_2\geq\cdots$.

 \subsection{Complete intersections of even number of quadrics and conjectural matching}Recall that the local systems $E_{i,2j}^{2n+1}$, where $ i\in[1,n]$ and $2j\in[0,i-1]$, are constructed from families of complete intersections $X_{2m}$ of even number of quadrics in $\bP^{2n}$ for $m\in[1,n]$. 
 
 We first show  that
\beq\label{support even}
 \text{$\mathfrak{F}(\IC(\Lg_1,E_{i,2j}^{2n+1}))$ is supported on an {\em even} dimensional $K$-orbit in $\cN_1^3$}\,.
\eeq
To this end we first note that each $\fF(\IC(\Lg_1,E_{i,2j}^{2n+1}))$ is a direct summand of $\fF(\IC(\Lg_1,P(X_{2m})))$ for some $m$, which in turn is a direct summand of $(\tau_m^N)_*\bC[-]$. One readily checks that 
\beqn
\dim\on{Hess}_m^{E} =m(4n-3m+5)-2n-2, \text{ which is even}.
\eeqn
Note also that $\dim X_{2m,a}$ is even. Now~\eqref{support even} follows from the decomposition theorem and the fact that the fibers of $\tau_m^N$ have non-vanishing cohomology only in even degrees (see \S\ref{affine pavings}).

Thus~\eqref{orbit dim} puts a restriction on nilpotent orbits which can support $\fF(\IC(\Lg_1,E_{i,2j}^{2n+1}))$.
Our first conjecture is:
\begin{conjecture}\label{conj 1}
We have that
\begin{eqnarray*}
&\mathfrak{F}(\IC(\Lg_1,E_{i,2j}^{2n+1}))\cong\IC({\bar\cO_{3^{2(n-i)+1}2^{2(i+j-n)-1}1^{2i-4j}}},\bC)\text{ if } i+j\geq n+1\\
&\mathfrak{F}(\IC(\Lg_1,E_{i,2j}^{2n+1}))\cong\IC({\bar\cO_{3^{2j}2^{2(n-i-j)+1}1^{4i-2n-2j-1}}},\bC)\text{ if } i+j\leq n\text{ and }2i-j\geq n+1
\\
&\mathfrak{F}(\IC(\Lg_1,E_{i,2j}^{2n+1}))\cong\IC({\bar\cO_{3^{2j}2^{2i-4j}1^{2n-4i+2j+1}}},\bC)\text{ if}\ i+j\leq n\text{ and }2i-j\leq n.
\end{eqnarray*}
\end{conjecture}
\begin{remark}
The nilpotent orbits appearing in the conjecture above exhaust all the non-zero {\em even} dimensional orbits of the form $\cO_{3^i2^j1^k}$, where the partition $3^i2^j1^k$ has no gaps.
\end{remark}

Note that the conjecture above holds for $E_{i,0}^{2n+1}$. This follows from \eqref{E and L} and \cite[Theorem 2.2]{CVX}, i.e., we have 
\beq\label{Ei0}
\begin{gathered}
\fF(\IC(\Lg_1,E_{i,0}^{2n+1}))=\IC(\bar\cO_{2^{2i}1^{2n-4i+1}},\bC)\ \ \emph{if}\ \ 2i\leq n\\
\fF(\IC(\Lg_1,E_{i,0}^{2n+1}))=\IC(\bar\cO_{2^{2n-2i+1}1^{4i-2n-1}},\bC)\ \ \emph{if}\ \ 2i\geq n+1.
\end{gathered}
\eeq

Below we verify the conjecture in a simple case that involves nilpotent orbits of order 3.

\subsection{Complete intersection of $4$ quadrics, $n\geq 3$.}In this subsection we show that
\beq\label{Fourier ex1}
\fF(\IC(\Lg_1,E_{n,2}^{2n+1}))=\IC(\bar\cO_{3^12^11^{2n-4}},\bC).
\eeq

Let us write $$\tau=\tau_2^{2n+1}:\on{Hess}_2^E\to\bar{\cO}_{3\,2\,1^{2n-4}},\ \ \check{\tau}=\check{\tau}_2^{2n+1}:\on{Hess}_2^{E,\p}\to\Lg_1.$$  
We have $\fF(\tau_*\bC[-])\cong\check{\tau}_*\bC[-]$ and
\beqn
\check{\tau}_*\bC[-]=\IC(\Lg_1,E_{n,2}^{2n+1}\oplus E_{n,0}^{2n+1}\oplus E_{n-1,0}^{2n+1})\oplus\bigoplus_{a=0}^{2n-4}\IC(\Lg_1,\bC)[2n-4-2a]\oplus\cdots
\eeqn
where $\cdots$ is a direct sum of IC complexes with smaller support. We have 
\beqn
\bar{\cO}_{3\,2\,1^{2n-4}}=\cO_{3\,2\,1^{2n-4}}\cup\cO_{3^11^{2n-2}}\bigcup_{0\leq i\leq 3}\cO_{2^i1^{2n-2i+1}}.
\eeqn
In view of Proposition \ref{matching nontrivial}, Lemma \ref{E and tE}, \eqref{tE and F} and \eqref{Ei0}, we conclude that $\fF(\IC(\Lg_1,E_{n,2}^{2n+1}))$ is not supported on $\bar\cO_{2^i1^{2n+1-2i}}$'s. Now it follows from \eqref{orbit dim} and \eqref{support  even} that 
$$\text{$\fF(\IC(\Lg_1,E_{n,2}^{2n+1}))$ is supported on $\bar\cO_{3\,2\,1^{2n-4}}$.}$$
Thus \eqref{Fourier ex1} follows the fact that the only IC complex supported on $\cO_{3\,2\,1^{2n-4}}$ appearing in $\tau_*\bC[-]$ is $\IC(\bar\cO_{3\,2\,1^{2n-4}},\bC)$ as $\tau$ is a resolution of $\bar\cO_{3\,2\,1^{2n-4}}$.

\subsection{Complete intersections of odd number of quadrics and a conjectural matching}Recall that the local systems $E_{i,2j-1}^{2n+1}$, where $ i\in[1,n]$ and $2j\in[2,i]$, are constructed from complete intersections $X_{2m-1}$ of odd number of quadrics in $\bP^{2n}$, $m\in[1,n]$. 

Using that $\dim\on{Hess}_m^{O}=m(2n-3m+5)-2n-3$, which is odd, and arguing as in \eqref{support even}, we obtain that
\beq\label{support odd}
\text{ $\mathfrak{F}(\IC(\Lg_1,E_{i,2j-1}^{2n+1}))$ is supported on an {\em odd} dimensional $K$-orbit in $\cN_1^3$.}
\eeq

Let $\cO\subset\cN_1^3$ be an odd dimensional $K$-orbit. To describe our second conjecture, let us first label the non-trivial irreducible $K$-equivariant local systems on $\cO$ as follows. By \eqref{orbit dim}, we can assume that $\cO=\cO_{3^{2k-1}2^{2l}1^{2n+4-6k-4l}}.$  

Let $x\in\cO_{3^{2k-1}2^{2l}1^{2n+4-6k-4l}}$, $k\geq 1$. We first define representatives for the component group $A_K(x)=Z_K(x)/Z_K(x)^0$. 
 Take a basis 
$$\text{$x^iu_j,i\in[0,2],j\in[1,2k-1]$, $x^iv_j,\ i\in[0,1],\ j\in[1,2l]$ and $w_i,\ i\in[1,2n+4-6k-4l]$}$$ of $V$ as in \cite[Lemma 5.6]{CVX}. Define $\gamma_i\in Z_K(x)$, $i=1,2$ as follows
\begin{eqnarray*}
&\gamma_1(w_1)=w_2,\ \gamma_1(w_2)=w_1,\ \gamma_1(x^iu_j)=-x^iu_j,\ i\in[0,2],\ j\in[1,2k-1],\\ &\gamma_2(x^jv_1)=x^jv_2,\ \gamma_2(x^jv_2)=x^jv_1,\ j\in[0,1],\\
 &\quad\text{ and }\gamma_1 (\text{resp. }\gamma_2)\text{ acts as identity on all other basis vectors}.
\end{eqnarray*}

Assume that $l\geq 1$ and $2n+4-6k-4l\neq 0$. Then 
$A_K(x)\cong\{1,\gamma_1,\gamma_2,\gamma_1\gamma_2\}\cong(\bZ/2\bZ)^2.$
Let $$\text{$\cE_{k,l}^1$ (resp. $\cE_{k,l}^2$, $\cE_{k,l}^3$)}$$ denote the irreducible $K$-equivariant local system on $\cO_{3^{2k-1}2^{2l}1^{2n+4-6k-4l}}$ corresponding to the irreducible character of $A_K(x)$ $$\text{$\chi_1$  (resp. $\chi_2$, $\chi_3$)  with $\chi(\gamma_1)=-1$ (resp. $-1,1$) and $\chi(\gamma_2)=1$ (resp. $-1,-1$).}$$ 

 Assume that $l=0$ and $2n+4-6k\neq 0$. Then $A_K(x)\cong\{1,\gamma_1\}\cong\bZ/2\bZ$. We denote by $\cE^1_{k,0}$ the irreducible $K$-equivariant local system on $\cO_{3^{2k-1}1^{2n+4-6k}}$ corresponding to the irreducible character $\chi$ of $A_K(x)$ with $\chi(\gamma_1)=-1$. 
 
 Assume that $l\geq 1$ and $2n+4-6k-4l\neq 0$. Then $A_K(x)\cong\{1,\gamma_2\}\cong\bZ/2\bZ$. We denote by $\cE^3_{k,l}$ the irreducible $K$-equivariant local system on $\cO_{3^{2k-1}2^{n+2-3k}}$ corresponding to the irreducible character $\chi$ of $A_K(x)$ with $\chi(\gamma_2)=-1$.

We will simply write $\cE^i$, $i=1,2,3$, when  the supports of these local systems are clear.

Our second conjecture is the following.
\begin{conjecture}\label{conj 2}
We have that
\begin{eqnarray*}
&\mathfrak{F}(\IC(\Lg_1,E_{i,2j-1}^{2n+1}))\cong\IC({\bar\cO_{3^{2(n-i)+1}2^{2(i+j-n-1)}1^{2i-4j+2}}},\cE^1)\text{ if } i+j\geq n+1\\
&\mathfrak{F}(\IC(\Lg_1,E_{i,2j-1}^{2n+1}))\cong\IC({\bar\cO_{3^{2j-1}2^{2(n-i-j+1)}1^{4i-2j-2n}}},\cE^2)\text{ if } i+j\leq n\text{ and }2i-j\geq n+1\\
&\mathfrak{F}(\IC(\Lg_1,E_{i,2j-1}^{2n+1}))\cong\IC({\bar\cO_{3^{2j-1}2^{2(i-2j+1)}1^{2n-4i+2j}}},\cE^3)\text{ if } i+j\leq n\text{ and }2i-j\leq n.
\end{eqnarray*}
\end{conjecture}
\begin{remark}
In particular, the conjecture above implies that the set of all Fourier transforms $\mathfrak{F}(\IC(\Lg_1,E_{i,2j-1}^{2n+1}))$ coincides with the set of all IC complexes supported on  {\em odd} dimensional orbits in $\cN_1^3$, with {\em non-trivial} local systems.
\end{remark}
 In the following subsections we verify the conjecture above in two simple examples, see \eqref{FT ex2} and  \eqref{FT ex3}. We also prove a lemma (Lemma \ref{lem curve case}) that is compatible with our conjecture.

\subsection{Complete intersection of 3 quadrics, $n\geq2$}In this subsection we show that
\beq\label{FT ex2}
\mathfrak F(\IC(\fg_1,E_{n,1}^{2n+1}))=\IC(\bar\cO_{3^11^{2n-2}},\cE^1).
\eeq
Let us write $$\sigma=\sigma_2^{2n+1}:\on{Hess}_2^{O}\to\bar{\cO}_{3^11^{2n-2}},\ \ \check{\sigma}=\check\sigma_2^{2n+1}:\on{Hess}_2^{O,\p}\to\Lg_1.$$
The fiber $\sigma^{-1}(x)$ at $x\in\cO_{3^11^{2n-2}}$
is a non-singular quadric in $\mathbb P^{2n-3}$. Thus in the decomposition of $\sigma_*\bC[-]$, we have the following direct summands
\beqn
\bigoplus_{a=0}^{2n-4}\IC(\bar\cO_{3^11^{2n-2}},\bC)[2n-4-2a]\oplus\IC(\bar\cO_{3^11^{2n-2}},\cE^1).
\eeqn
We have $\fF(\sigma_*\bC[-])\cong\check{\sigma}_*\bC[-]$ and
\beqn
\check{\sigma}_*\bC[-]\cong\IC(\Lg_1,E_{n,1}^{2n+1})\oplus\cdots
\eeqn
Note that $\cO_{3^11^{2n-2}}$ is the only odd-dimensional orbit contained in 
$\bar{\cO}_{3^11^{2n-2}}$ and there is a unique non-trivial irreducible $K$-equvariant local system on $\cO_{3^11^{2n-2}}$, denoted by $\cE^1$. In view of \eqref{support odd}, the equation \eqref{FT ex2} follows from the fact that the support of 
$\fF(\IC(\bar\cO_{3^11^{2n-2}},\bC))$ is a proper subset of $\fg_1$ (see \cite[Proposition 4.4]{CVX2}).

\begin{remark}
Here we see that Fourier trxansform of $\IC$ complexes supported on
nilpotent orbits $\cO_\lambda$, where $\lambda$ has \emph{gaps},  with  
nontrivial local systems  can have full support (compare with \cite[Corollary 4.9]{CVX2}).
\end{remark}

\subsection{Complete intersection of 5 quadrics, $n\geq 4$}In this subsection, we show that
\beq\label{FT ex3}
\begin{gathered}
\mathfrak{F}(\IC(\bar\cO_{3^12^21^{2n-6}},\cE^1))=\IC(\Lg_1,E_{n,3}^{2n+1}),\ \ \mathfrak{F}(\IC(\bar\cO_{3^12^21^{2n-6}},\cE^2))=\IC(\Lg_1,E_{n-1,1}^{2n+1}),\\
\mathfrak{F}(\IC(\bar\cO_{3^12^21^{2n-6}},\cE^3))=\IC(\Lg_1,E_{?,1}^{2n+1}).
\end{gathered}
\eeq

Let us write $$\sigma=\sigma_3^{2n+1}:\on{Hess}_3^{O}\to\bar{\cO}_{3^21^{2n-5}},\ \check\sigma=\check\sigma_3^{2n+1}:\on{Hess}_3^{O,\p}\to\Lg_1.$$ We have $\fF(\sigma_*\bC[-])\cong\check{\sigma}_*\bC[-]$ and
\beqn
\check{\sigma}_*\bC[-]\cong\IC(\Lg_1,E_{n,1}^{2n+1}\oplus E_{n,3}^{2n+1}\oplus E_{n-1,1}^{2n+1})\oplus\cdots
\eeqn
The odd dimensional orbits contained in  $\Img\sigma=\bar{\cO}_{3^21^{2n-5}}$ are $\cO_{3^12^21^{2n-6}}$ and $\cO_{3^11^{2n-2}}$. In view of \eqref{FT ex2}, the equation \eqref{FT ex3} follows from Lemma \ref{lem curve case} (see \S\ref{ssec curve case}) and the following statement.
\beq\label{ICsupp}
\begin{gathered}
\text{The IC complexes supported on $\cO_{3^12^21^{2n-6}}$, that appear in the decomposition}\\\text{ of $\sigma_*\bC[-]$,  are
$
\IC({\bar\cO_{3^12^21^{2n-6}}},\cE^1\oplus\cE^2).
$}
\end{gathered}
\eeq
It remains to prove \eqref{ICsupp}. 
Note that there is no orbit $\cO$ such that $\cO_{3^12^21^{2n-6}}<\cO<\cO_{3^21^{2n-5}}$. The fiber $\sigma^{-1}(x_2)$ at $x_2\in\cO_{3^21^{2n-5}}$  is a nonsingular quadric in $\bP^{2n-6}$. Thus in the decomposition of $\sigma_*\bC[-]$,
\beqn
 \text{the IC complexes supported on $\cO_{3^21^{2n-5}}$ are $\bigoplus_{a=0}^{2n-7}\IC({\bar\cO_{3^21^{2n-5}}},\bC)[2n-7-2a]$.}
 \eeqn The fiber at $x_1\in\cO_{3^12^21^{2n-6}}$ is a quadric bundle over $\pi_0^{-1}(x_1)$ with fibers being a quadric $Q$ of rank $2n-8$ in $\bP^{2n-6}$. Here $\pi_0$ is Reeder's resolution of $\bar{\cO}_{3^21^{2n-5}}$, i.e.
$$\pi_0:\{(x,0\subset V_2\subset V_2^\p\subset V)\,|\,x\in\Lg_1,\,xV_2=0,xV_2^\p\subset V_2\}\to\bar{\cO}_{3^21^{2n-5}}.$$ It is easy to check that the map $\pi_0$ is small. Thus we have
\beqn
\mH^{k-8(n-1)}_{x_1}\IC(\bar\cO_{3^21^{2n-5}},\bC)=H^k(\pi_0^{-1}(x_1),\bC).
\eeqn
Note that $H^{\text{odd}}(\pi_0^{-1}(x_1),\bC)=0$, $H^{\text{odd}}(\sigma^{-1}(x_1),\bC)=0$, and 
\begin{eqnarray*}
&H^{2k}(\sigma^{-1}(x_1),\bC)=\bigoplus_{a=0}^{2n-7}H^{2a}(Q,\bC)\otimes H^{2k-2a}(\pi_0^{-1}(x_1),\bC)\\
&\cong\bigoplus_{a=0}^{2n-7}H^{2k-2a}(\pi_0^{-1}(x_1),\bC)\oplus (H_{\on{prim}}^{2n-6}(Q,\bC)\otimes H^{2k-2n+6}(\pi_0^{-1}(x_1),\bC)).
\end{eqnarray*}
We have $\on{codim}_{\on{Hess}_3^O}\cO_{3^12^21^{2n-6}}=2n-6$ and $\pi_0^{-1}(x_1)$ consists of two points. Moreover, $A_K(x_1)$ acts on $H_{\on{prim}}^{2n-6}(Q,\bC)\otimes H^{2k-2n+6}(\pi_0^{-1}(x_1),\bC)$ as $\chi_1(1\oplus\chi_3)=\chi_1\oplus\chi_2$. The equation \eqref{ICsupp} follows. This finishes the proof of \eqref{FT ex3}.

\begin{remark}
 Note that \eqref{FT ex3} shows that all three IC complexes supported on $\cO_{3^12^21^{2n-6}}$ with {\em nontrivial} local systems correspond to the monodromy representations constructed   from complete intersections of {\em odd} number of quadrics.
 \end{remark}

\subsection{The case of a curve}\label{ssec curve case}In this subsection we prove the following lemma by considering the family $X_{2n-1}$ of complete intersections of quadrics in $\bP^{2n}$.

\begin{lemma}\label{lem curve case}
For each $i\in[1,n-1]$, there exists some $1\leq j\leq [\frac{n-1}{2}]$ and a nontrival local system $\cE_{j}^s$ ($s=2\text{ or }3$) on $\cO_{3^12^{2j}1^{2n-2j-2}}$ such that
\beqn
\fF(\IC(\Lg_1,E^{2n+1}_{i,1}))\cong\IC(\bar\cO_{3^12^{2j}1^{2n-4j-2}},\cE_{j}^s).
\eeqn
\end{lemma}
Let us write 
$$\sigma=\sigma_n^{2n+1}:\on{Hess}_n^{O}\to\bar{\cO}_{3^12^{n-1}}\ \text{ and } \check\sigma=\check\sigma_n^{2n+1}:\on{Hess}_n^{O,\p}\to\Lg_1.$$ Assume that $n\geq 3$.
We have 
$
 \dim \on{Hess}_n^{O}=n^2+3n-3.
$
We show that
\begin{eqnarray}\label{decompositiontau}
&&\sigma_{*}\bC[-]\cong\bigoplus_{a=0}^{n-3}\IC({\bar\cO_{3^12^{n-1}}},\bC)[n-3-2a]\bigoplus _{2j=n-1}\IC({\bar\cO_{3^12^{2j}}},\cE^3)\nonumber\\
&&\qquad\oplus\bigoplus_{2\leq 2j\leq n-2}\IC({\bar\cO_{3^12^{2j}1^{2n-4j-2}}},\cE^2\oplus\cE^3)\oplus \IC({\bar\cO_{3^11^{2n-2}}},\bC\oplus\cE^1)\\
&&\qquad\oplus \IC(\bar\cO_{1^{2n+1}},\bC)[1]\oplus \IC(\bar\cO_{1^{2n+1}},\bC)[-1].\nonumber
\end{eqnarray}
The lemma follows from the decomposition above, the equations $\fF(\check\sigma_*\bC[-])\cong\sigma_*\bC[-]$,
$\check\sigma_*\bC[-]\cong\bigoplus_i\IC(\Lg_1,E_{i,1}^{2n+1})\oplus\cdots$, and \eqref{FT ex2}.

In the remainder of this subsection we prove \eqref{decompositiontau}. 
Consider first Reeder's resolution of $\bar{\cO}_{3^12^{n-1}}$ given by
\beqn
\rho:\{(x,0\subset V_1\subset V_n\subset V_n^\p\subset V_1^\p\subset \bC^{2n+1})\,|\,x\in\Lg_1,\,xV_n=0,xV_n^\p\subset V_1\}\to\bar{\cO}_{3^12^{n-1}}.
\eeqn
It is easy to check that $\rho$ is a small map.
Thus for $x_j\in\cO_{3^12^j1^{2n-2j-2}}$, we have 
\beq\label{stalkn1}
\mH_{x_j}^{k-n^2-2n}\IC(\bar\cO_{3^12^{n-1}},\bC)\cong H^k(\rho^{-1}(x_j),\bC).
\eeq

Now we study the map $\sigma$ and the decomposition of $\sigma_*\bC[-]$. The fiber $\sigma^{-1}(x_{n-1})$ at $x_{n-1}\in\cO_{3^12^{n-1}}$ is a nonsingular quadric in $\bP^{n-2}$ and $\on{codim}_{\on{Hess}_n^{O}}\cO_{3^12^{n-1}}=n-3$. Thus in the decomposition of $\sigma_*\bC[-]$, the following IC complexes supported on $\cO_{3^12^{n-1}}$ appear,
\beq\label{decom-ds}
\bigoplus_{a=0}^{n-3}\IC(\bar\cO_{3^12^{n-1}},\bC)[n-3-2a]\text{ for all }n\text{ and }
\IC(\bar\cO_{3^12^{n-1}},\cE)[-]\text{ if }n\text{ is  odd},
\eeq
where $\cE$ is the unique nontrivial irreducible $K$-equivariant  local system on $\cO_{3^12^{n-1}}$.
 
We have that $\sigma^{-1}(x_j)$ ($x_j\in\cO_{3^12^j1^{2n-2j-2}}$) is a $Q_j$-bundle over $\rho^{-1}(x_j)$  for $j\geq 1$, where $Q_j$ is a quadric $\sum_{k=1}^ja_k^2=0$ in $\bP^{n-2}:=\{[a_1,\cdots,a_{n-1}]\}$. 

For $j$ odd, or $j\geq 2$ even and $2k>\on{codim}_{\on{Hess}_n^{O}}\cO_{x_j}$, we have
\beqn
H^{2k}(\sigma^{-1}(x_j),\bC)\cong\bigoplus_{a=0}^{n-3}H^{2k-2a}(\rho^{-1}(x_j),\bC)\cong\mH_{x_j}^{2k}\bigoplus_{a=0}^{n-3}\IC({\bar\cO_{3^12^{n-1}}},\bC)[-n^2-2n-2a]
\eeqn
where in the second isomorphism we use \eqref{stalkn1}. Thus in view of \eqref{decom-ds} IC complexes supported on $\cO_{3^12^{j}1^{2n-2j-2}}$, for odd $j<n-1$, do not appear in the decomposition of $\sigma_*\bC$.

For $j\geq 2$ even, and $2k=\on{codim}_{\on{Hess}_n^{O}}\cO_{x_j}$, we have
\beqn
H^{2k}(\sigma^{-1}(x_j),\bC)\cong\mH_{x_j}^{2k-n^2-2n}\bigoplus_{a=0}^{n-3}\IC({\bar\cO_{3^12^{n-1}}},\bC)[-2a]
\eeqn
\beqn
\oplus (H^{2n+j-4}_{\text{prim}}(Q_j,\bC)\otimes H^{2\dim\rho^{-1}(x_j)}(\rho^{-1}(x_j),\bC)).
\eeqn
Note that $\rho^{-1}(x_j)$ has two irreducible components. Moreover $A_K(x_j)$ acts on $H^{2n+j-4}_{\text{prim}}(Q_j,\bC)$ via the character $\chi_3$, and acts on $H^{2\dim\rho^{-1}(x_j)}(\rho^{-1}(x_j),\bC)$ via $1\oplus\chi_1$.
In view of \eqref{decom-ds}, we conclude that IC complexes $\IC(\bar\cO_{3^12^j1^{2n-2j-2}},\cE^2)$ and $\IC(\bar\cO_{3^12^j1^{2n-2j-2}},\cE^3)$, for $j$ even, appear in the decomposition of $\sigma_*\bC[-]$. 

For $j=0$ and $2k=\text{codim}_{\on{Hess}_n^{O}}\cO_{x_0}=n^2-n-2$, since $2k-2a>2\dim\rho^{-1}(x_0)$ for all $0\leq a\leq n-3$, we have
\beqn
\bigoplus_{a=0}^{n-3}\mH_{x_0}^{2k-2a-n^2-2n}\IC({\bar\cO_{3^12^{n-1}}},\bC)=0.
\eeqn
We have $2\dim\sigma^{-1}(x_0)={\on{codim}_{\on{Hess}_n^{O}}{\cO_{x_0}}}$ and $\sigma^{-1}(x_0)\cong\{0\subset W_{n-2}\subset W_{n-1}\subset W_{n-1}^\p\subset W_{n-2}^\p\subset\bC^{2n-2}\}$. Note that $\sigma^{-1}(x_0)$ has two connected components and $A_k(x_0)$ permutes them. We conclude that the IC complexes supported on $\cO_{x_0}$ appearing in $\sigma_*\bC[-]$ are $\IC(\bar\cO_{3^11^{2n-2}},\bC)\oplus\IC(\bar\cO_{3^11^{2n-2}},\cE^1)$.

The decomposition \eqref{decompositiontau} follows from the above discussion and the fact that none of the IC complexes supported on $\cO_{2^i1^{2n+1-2i}}$, $i\geq 1$ can appear in the decomposition of $\sigma_*\bC[-]$.
The proof of Lemma \ref{lem curve case} is complete.

\end{document}